\documentclass[
  a4paper, 
  reqno, 
  oneside, 
  12pt
]{amsart}

\usepackage[utf8]{inputenc}

\usepackage[dvipsnames]{xcolor} 
\colorlet{cite}{red}
\usepackage{tikz}  
\usetikzlibrary{arrows,positioning}
\tikzset{ 
  baseline=-2.3pt,
  text height=1.5ex, text depth=0.25ex,
  >=stealth,
  node distance=2cm,
  mid/.style={fill=white,inner sep=2.5pt},
}

\usepackage{lmodern}
\usepackage{tikz-cd}
\usepackage{amsthm, amssymb, amsfonts}

\usepackage{amsthm, amssymb, amsfonts}	
\usepackage[all]{xy}
\usepackage{graphicx,caption,subcaption}
\usepackage{braket}  
\usepackage{microtype}
\usepackage[makeroom]{cancel}
\usepackage[%
  bookmarks=true,			
  unicode=true,			
  pdftitle={Generalized K\"ahler structures},		%
  pdfauthor={Valencia}{Varea}{ Gasparim},	%
  pdfkeywords={Invariant generalized complex structures},	
  colorlinks=true,		
  linkcolor=Red,			
  citecolor=cite,		
  filecolor=magenta,		
  urlcolor=RoyalBlue			
]{hyperref}				
\usepackage{cleveref}

\newtheoremstyle{mydef}
  {}		
  {}		
  {}		
  {}		
  {\scshape}	
  {. }		
  { }		
  {\thmname{#1}\thmnumber{ #2}\thmnote{ #3}}	

\newtheorem{theorem}{Theorem}[section]
\newtheorem*{theorem*}{Theorem}
\newtheorem{proposition}[theorem]{Proposition}
\newtheorem*{proposition*}{Proposition}
\newtheorem{lemma}[theorem]{Lemma}
\newtheorem*{lemma*}{Lemma}
\newtheorem{corollary}[theorem]{Corollary}
\newtheorem*{corollary*}{Corollary}
\theoremstyle{definition}
\newtheorem{definition}[theorem]{Definition}
\newtheorem{example}[theorem]{Example}
\newtheorem{notation}[theorem]{Notation}
\theoremstyle{remark}
\newtheorem{remark}[theorem]{Remark}

\usepackage{tikz}

\DeclareMathOperator{\Ad}{Ad}

\DeclareMathOperator{\Ann}{Ann}
\newcommand{\ce}{\mathrel{\mathop:}=}

\author{Elizabeth Gasparim, Fabricio Valencia, Carlos Varea}
\subjclass[2010]{53D18; 14M15}
\address{}
\date{\today}
\address{E. Gasparim - Depto. Matem\'aticas, Univ. Cat\'olica del Norte, Chile. \newline
		 F. Valencia - Instituto de Matem\'aticas, Univ. de Antioquia, Colombia. \newline
		 C. Varea - Imecc -	Unicamp,  Campinas, Brasil.\newline
	etgasparim@gmail.com,  fabricioyarro@gmail.com, carlosbassanivarea@gmail.com}
\title[Invariant generalized complex geometry on maximal flag manifolds]{Invariant generalized complex geometry on maximal flag manifolds and their moduli}

\begin{document}
\maketitle

\begin{abstract}
We describe moduli spaces  of invariant generalized  complex structures 
and moduli spaces of  invariant generalized K\"ahler structures
on maximal flag manifolds under $B$-transformations. We give an alternative description of the moduli space of generalized 
 complex structures  using  pure spinors, and describe a cell decomposition of these moduli spaces induced by the action of the Weyl group.
\end{abstract}

\tableofcontents
\section{Introduction}
Generalized complex geometry is a theory that allows us to treat both complex  and symplectic geometry within a single framework. 
This theory was first introduced in Hitchin's  seminal paper \cite{H} where the notion of generalized complex structure was presented 
as a geometric structure on even-dimensional manifolds which allows for generalizations of the notions of Calabi--Yau manifold, symplectic manifold, and K\"ahler manifold. Further results  were then developed by 
 Gualtieri \cite{G3} and Cavalcanti \cite{Ca2}.
Currently, generalized complex geometry plays an important role in applications to  various aspects of  string theory, see for instance \cite{Gr} and \cite{BLPZ},
and to problems related to mirror symmetry and T-duality \cite{CG2}. 
 
The aim of this paper is to study some geometric properties of all invariant generalized complex structures on maximal flag manifolds which were determined by the third author and San Martin in \cite{VS}. Other examples of generalized complex structures constructed using Lie theory can be found in \cite{CG} were the authors gave a classification of all 6-dimensional nilmanifolds admitting generalized complex structures and in \cite{AD} were the authors described a regular class of invariant generalized complex structures on a real semisimple Lie group.

Let $G$ be a semisimple Lie group with Lie algebra $\mathfrak{g}$. A {\it maximal flag manifold} of $G$ 
 is a homogeneous space $\mathbb{F}=G/P$ where $P$ is a Borel subgroup  of $G$. If $U$ is a compact real form in $G$ then $U$ acts transitively on $\mathbb{F}$ so that we can also identify the homogeneous space with $\mathbb{F}=U/T$ where $T=P\cap U$ is a maximal torus in $U$. In this paper we will deal with generalized complex structures on $\mathbb{F}$ that are $U$-invariant with respect the adjoint action. Because 
 of invariance, these structures are completely determined by their values at the origin $b_0$ of $\mathbb{F}$. Thus, every invariant generalized almost complex structure is determined by an orthogonal complex structure $\mathbb{J}:\mathfrak{n}^-\oplus(\mathfrak{n}^-)^\ast\to \mathfrak{n}^-\oplus(\mathfrak{n}^-)^\ast$ that commutes with  the adjoint action of $T$ on $\mathfrak{n}^-\oplus(\mathfrak{n}^-)^\ast$. 
 Here we  denote by $\mathfrak{n}^-\ce T_{b_0}\mathbb{F}=\sum_{\alpha\in\Pi^+}\mathfrak{u}_\alpha$ where $\mathfrak{u}_\alpha=(\mathfrak{g}_\alpha\oplus\mathfrak{g}_{-\alpha})\cap \mathfrak{u}$ with $\mathfrak{u}$ the Lie algebra of $U$ and $\mathfrak{g}_\alpha$ the root space associated to $\alpha\in \Pi^+$.

The paper is divided as follows. In section \ref{S:2} we review some elementary facts about generalized complex geometry which 
we will use throughout the paper. In section \ref{igcs} we introduce 
 invariant generalized complex structures on maximal flags. For each root space, the space of such structures consists of 
  a 2-dimensional family of structures of noncomplex type parametrized by a real algebraic surface in $\mathbb{R}^3$ plus two extra points corresponding to 2 structures of complex type. We summarize the necessary and sufficient  conditions found in \cite{VS} for  an invariant generalized almost complex structure
   to be integrable. In section \ref{S:4} we describe the natural action of the Weyl group associated to a root system on the space of all invariant generalized almost complex structures on a maximal flag. We show that the integrability condition is preserved by this Weyl action; see Proposition \ref{WeylIntegrability}. 

Section \ref{S:5} is devoted to studying the effects of the action by $B$-transformations on the space of invariant generalized almost complex structures. 
 We prove that, for every root space, any generalized complex structure which is not of complex type, is a 
$B$-transform of a structure of symplectic type.
We also show that the two generalized complex structures of complex type are fixed points for the action induced by $B$-transformations. Adapting these results to the general case and denoting by $\mathfrak{M}_a(\mathbb{F})$ the quotient space obtained from the set of all invariant generalized almost complex structures modulo the action by $B$-transformations, we obtain:

\begin{theorem*}[\ref{Bmoduli1}]
	Suppose that $\vert \Pi^+ \vert=d$. Then
	$$\mathfrak M_a (\mathbb F)= \prod_{\alpha\in\Pi^+} \mathfrak{M}_\alpha (\mathbb F)=(\mathbb R^\ast \cup \pm \mathbf{0})_{\alpha_1} \times \cdots \times (\mathbb  R^\ast\cup\pm \mathbf{0})_{\alpha_d},$$
	where $\alpha_j\in\Pi^+$.  Moreover, $\mathfrak{M}_a(\mathbb{F})$ admits a natural topology such that it contains an open dense subset $ (\mathbb{R}^\ast)^d$ parametrizing all invariant generalized almost complex structures of symplectic type and there are exactly $2^d$ isolated points corresponding to the structures of complex type. 
\end{theorem*}
For every subset $\Theta$ of a simple root system $\Sigma$ of $\mathfrak{g}$ with $\vert \Theta \vert=r$, we will show a natural way of defining an $r$-dimensional family inside $\mathfrak M(\mathbb F)$ of invariant generalized complex structures on $\mathbb{F}$ of type $k=d-\vert\langle\Theta \rangle^+\vert$ (see Lemma \ref{ThetaCells}). This will be called a $\Theta$-cell and denoted by ${\bf c}^r$. If $\mathcal{W}$ denotes the Weyl group and for every $w\in\mathcal{W}$ we denote by $w\cdot {\bf c}^i$ the $\Theta$-cell of dimension $i$ determined by $w\cdot \Theta$, we prove:
\begin{theorem}[\ref{IntModuli}]
	The moduli space of invariant generalized complex structures on $\mathbb F$ admits the following  cell decomposition:
	$$\mathfrak M (\mathbb F)=\bigsqcup_{i=1}^{d} \bigsqcup_{w\in \mathcal W} w\cdot {\bf c}^i.$$
\end{theorem}

Let $\omega$ denote the Kirillov--Kostant--Souriau symplectic form on $\mathbb{F}$. Let $\mathbb{J}$ be an invariant generalized almost complex structure on  $\mathbb{F}$. We decompose the set of positive roots as $\Pi^+= \Theta_{nc}\sqcup\Theta_{c}$ where $\mathbb{J}$ restricted to a root space is of noncomplex type for all $\alpha\in \Theta_{nc}$ and of complex type for all $\alpha\in \Theta_{c}$. We get:

\begin{proposition*}[\ref{Spinor}]
The invariant pure spinor line $K_\mathcal{L}<\bigwedge^\bullet (\mathfrak{n}^-)^\ast\otimes \mathbb{C}$ associated to $\mathbb{J}$ is generated by
	$$ \varphi_\mathcal{L}=e^{\sum_{\alpha\in \Theta_{nc}}(B_\alpha +i\omega_\alpha)}\bigwedge_{\alpha\in \Theta_{c}}\Omega_\alpha,$$
where
\begin{itemize}		 
\item 	$B_\alpha=-\dfrac{a_\alpha}{x_\alpha}S_\alpha^\ast\wedge A_\alpha^\ast$ and  $\omega_\alpha=-\dfrac{k_\alpha}{x_\alpha}\omega_{b_0}|_{\mathfrak{u}_\alpha}$ with $k_\alpha=\dfrac{1}{2i\langle H,H_\alpha\rangle}$ for all $\alpha\in \Theta_{nc}$, and
\item  $\Omega_\alpha\in \wedge^{1,0}\mathfrak{u}_\alpha^\ast$ defines a complex structure on $\mathfrak{u}_\alpha$ for all $\alpha\in \Theta_{c}$.
\end{itemize}
\end{proposition*}

An invariant pure spinor
 $\widetilde{\varphi_\mathcal{L}}\in \bigwedge^k T^\ast\mathbb{F}\otimes \mathbb{C}$ determined by $\mathbb{J}$ can be defined as
$$\widetilde{\varphi_\mathcal{L}}_x(X_1(x),\cdots, X_k(x))=\varphi_\mathcal{L}(\Ad(g^{-1})_{\ast,x}X_1(x),\cdots,\Ad(g^{-1})_{\ast,x}X_k(x)),$$
for $x=gT\in \mathbb{F}=U/T$ with $g\in U$.

In section \ref{S:6} we classify all invariant generalized K\"ahler structures on $\mathbb{F}$. We show that, for every root space, a generalized K\"ahler structure is given by a pair formed by a structure of complex type together with the $B$-transformation of a structure of symplectic type with the same orientation (see Proposition  \ref{Kahlerroots}). This allows us to prove that, in general:

\begin{theorem*}[\ref{KahlerClassification}]
Let $(\mathbb{J},\mathbb{J} ')$ be an invariant generalized Kähler structure on $\mathbb{F}$. If $\mathbb{J}_{\alpha}$ is of noncomplex type for some positive root $\alpha$, then $\mathbb{J}_{\beta}$ is of noncomplex type for every positive root $\beta$. In particular, every invariant generalized Kähler structure on $\mathbb{F}$ is always obtained from a complex structure together with a $B$-transformation of a symplectic structure.
\end{theorem*}

We denote by $\mathfrak{K}_a(\mathbb{F})\subset \mathfrak{M}_a(\mathbb{F})\times \mathfrak{M}_a(\mathbb{F})$ the set of all invariant generalized almost K\"ahler structures on $\mathbb{F}$ module the diagonal action by $B$-transformations. If $\mathbb{R}^\dag$ denotes the union of the four rays $(\lbrace \mathbf{0}\rbrace \times \mathbb{R}^+)\cup (\mathbb{R}^+\times \lbrace \mathbf{0}\rbrace)\cup (\lbrace -\mathbf{0}\rbrace \times \mathbb{R}^-)\cup (\mathbb{R}^-\times \lbrace -\mathbf{0}\rbrace)$ we have that:
\begin{theorem*}[\ref{BmoduliK2}]
	Suppose that $\vert \Pi^+ \vert=d$. Then
	$$\mathfrak K_a (\mathbb F)= \prod_{\alpha\in\Pi^+} \mathfrak{K}_\alpha(\mathbb{F})=\mathbb{R}^\dag_{\alpha_1}\times \cdots \times\mathbb{R}^\dag_{\alpha_d}=(\mathbb{R}^\dag)^d.$$
	In particular, up to action by the Weyl group $\mathcal{W}$:
	$$\mathfrak K_a (\mathbb F)/\mathcal{W}=((\lbrace \mathbf{0}\rbrace \times \mathbb{R}^+)\cup (\mathbb{R}^+\times \lbrace \mathbf{0}\rbrace))_{\alpha_1}\times \cdots \times ((\lbrace \mathbf{0}\rbrace \times \mathbb{R}^+)\cup (\mathbb{R}^+\times \lbrace \mathbf{0}\rbrace))_{\alpha_d}.$$
\end{theorem*}
The topology on $\mathfrak K_a (\mathbb F)$ inherited from $\mathbb R^n$ coincides  with the   topology induced from the product 
$\mathfrak{M}_a(\mathbb{F})\times \mathfrak{M}_a(\mathbb{F})$. Let $\mathfrak{G}$ denote the set of all invariant generalized metrics on $\mathbb{F}$ module the action by $B$-field transformations. So, we show that: 
\begin{proposition*}[\ref{GereralizedMetrics}]
	Suppose that $\vert \Pi^+ \vert=d$. Then
	$$\mathfrak G (\mathbb F)= \prod_{\alpha\in\Pi^+} \mathfrak{G}_\alpha(\mathbb{F})=\mathbb{R}^\ast_{\alpha_1}\times \cdots \times\mathbb{R}^\ast_{\alpha_d}=(\mathbb{R}^\ast)^d.$$
	Moreover, up to action by the Weyl group $\mathcal{W}$, we have that $\mathfrak G (\mathbb F)/\mathcal{W}=(\mathbb{R}^+)^d$.
\end{proposition*}
Finally, in section \ref{S:7} we give some instructive examples using the root space decomposition of 
$\mathfrak{sl}(n,\mathbb{C})$.

\section{Background on generalized complex geometry}\label{S:2}
In what follows we review some fundamental facts about generalized complex geometry which we will use throughout the paper. The notions that we will introduce can be found in \cite{H}, \cite{G1}, \cite{G2}.

\begin{definition}\label{inn} Given an $n$-dimensional vector space $V$ over $\mathbb{R}$, the space $V\oplus V^\ast$ can be equipped 
with the {\it indefinite inner product} $\langle\cdot,\cdot\rangle$ of signature $(n,n)$ given by
\begin{equation}\label{innerproduct1}
\langle X+\xi,Y+\eta \rangle =\dfrac{1}{2}(\xi(Y)+\eta(X)).
\end{equation}
A subspace $\mathcal{L}\subset V\oplus V^\ast$ is called {\it maximal isotropic} if its dimension is $n$ and $\langle v,u \rangle=0$ for all $v,u\in \mathcal{L}$. 
\end{definition}

If $\Delta$ is a vector subspace of $V$ and $\varepsilon\in \bigwedge^2\Delta^\ast$, it is easy to check that regarding $\varepsilon$ as a skew map $\varepsilon:\Delta\to \Delta^\ast$ via $X\mapsto i_X\varepsilon$, the space
$$\mathcal{L}(\Delta,\varepsilon)=\{ X+\xi \in \Delta\oplus V^\ast:\ i^\ast(\xi)=\varepsilon(X) \},$$
is  a maximal isotropic subspace of $V\oplus V^\ast$. Actually, every maximal isotropic subspace of
 $V\oplus V^\ast$ is of the form $\mathcal{L}(\Delta,\varepsilon)$. Indeed, if $\mathcal{L}\subset V\oplus V^\ast$ is  maximal isotropic then $\mathcal{L}=\mathcal{L}(\Delta,\varepsilon)$ where $\Delta=\pi_1(\mathcal{L})$ and $\varepsilon:\Delta\to\Delta^\ast$ is defined as $X\mapsto \pi_2(\pi_1^{-1}(X)\cap \mathcal{L})$. Here $\pi_1$ and $\pi_2$ denote the projections from $V\oplus V^\ast$ onto $V$ and $V^\ast$, respectively. To verify the last affirmations 
  notice that since $\mathcal{L}$ is  
maximal isotropic $\Ann (\Delta)=\mathcal{L}\cap V^\ast$ and $\Delta^\ast= V^\ast/\Ann (\Delta)$.

The integer $k=\textnormal{dim Ann}(\Delta)=n-\textnormal{dim}\pi_1(\mathcal{L})$ is an invariant associated to any maximal isotropic subspace of $V\oplus V^\ast$.
\begin{definition}
	The \emph{type} of a maximal isotropic subspace $\mathcal{L}\subset V\oplus V^\ast$ is the codimension $k$ of its projection onto $V$.
\end{definition}
The Lie algebra of infinitesimal isometries of $(V\oplus V^\ast,\langle\cdot,\cdot\rangle)$, denoted  $\mathfrak{so}(V\oplus V^\ast)$, may be identified with the space $\mathfrak{so}(V\oplus V^\ast)=\bigwedge^2(V\oplus V^\ast)=\textnormal{End}(V)\oplus \bigwedge^2V^\ast\oplus\bigwedge^2V$. Consequently, every element $B\in \bigwedge^2V^\ast$ can be viewed as $ {\tiny \left( 
\begin{array}{cc}
0 & 0\\
B & 0
\end{array}%
\right) }\in \mathfrak{so}(V\oplus V^\ast)$, and then we  consider
\begin{equation}\label{Btransform}
e^B=\exp(B)=\left( 
\begin{array}{cc}
1 & 0\\
B & 1
\end{array}%
\right)\in SO(V\oplus V^\ast).
\end{equation}
This is an orthogonal transformation sending $X+\xi\mapsto X+\xi+i_XB$, 
called  a {\it $B$-transformation}. As these maps preserve the projections to $V$, they do not affect $\Delta$ and 
\begin{equation}\label{TypeNChange}
\exp(B)\mathcal{L}(\Delta,\varepsilon)=\mathcal{L}(\Delta, \varepsilon+i^\ast B),
\end{equation}
where $i:\Delta \hookrightarrow V$ is the natural subspace inclusion. Therefore, $B$-transformations do not change the type of a maximal isotropic subspace.

The action of $V\oplus V^\ast$ on $\bigwedge^\bullet V^\ast$ given by
\begin{equation}\label{CliffordAction}
(X+\xi)\cdot \varphi=i_X\varphi+\xi\wedge \varphi
\end{equation}
extends to a spin representation of the Clifford algebra $\textnormal{CL}(V\oplus V^\ast)$ associated to the natural indefinite inner product $\langle\cdot,\cdot\rangle$. 
Accordingly, we refer to the elements of $\bigwedge^\bullet V^\ast$ as {\it spinors}.
\begin{definition}\label{purespinor}
	A spinor $\varphi$ is {\it pure} when its null space
	$$\mathcal{L}_\varphi=\{ X+\xi\in V\oplus V^\ast:\ (X+\xi)\cdot \varphi=0\},$$ 
	is maximal isotropic. The real subspace $K_{\mathcal L}$ generated by  $\varphi$ over $\mathbb R$ in  $\bigwedge^\bullet V^\ast$ is called the {\it pure spinor line} generated by $\varphi$.
\end{definition}

Identity \eqref{TypeNChange} implies that every maximal isotropic subspace $\mathcal{L}(\Delta,\varepsilon)$ may be expressed as the $B$-transform of $\mathcal{L}(\Delta,0)$ for some $B$ chosen so that $i^\ast B=-\varepsilon$. The pure spinor line with null space $\mathcal{L}(\Delta,0)$ is precisely $\textnormal{det}(\Ann(\Delta))\subset \bigwedge^kV^\ast$ where  $k$ is the codimension of $\Delta\subset V$.
\begin{proposition}\cite{C}\label{Chevalley}
	Let $\mathcal{L}(\Delta,\varepsilon)\subset V\oplus V^\ast$ be a maximal isotropic subspace, let $(\theta_1,\dots,\theta_k)$ be a basis for $\Ann (\Delta)$, and let $B\in \bigwedge^2V^\ast$ be a 2-form such that $i^\ast B=-\varepsilon$, where $i:\Delta \hookrightarrow V$ is the natural subspace inclusion. Then the pure spinor line $K_\mathcal{L}$ representing $\mathcal{L}(\Delta,\varepsilon)$ is generated by
	$$\varphi_\mathcal{L}=\exp (B)\, \theta_1\wedge\cdots\wedge \theta_k.$$
\end{proposition}

These notions extend to the complex case, that is, when $\mathcal{L}\subset (V\oplus V^\ast)\otimes \mathbb{C}$
 is a maximal isotropic subspace  with respect to the indefinite inner product $\langle\cdot,\cdot \rangle$ extended to $(V\oplus V^\ast)\otimes \mathbb{C}$; see \cite{G1}.

Let us now consider an $n$-dimensional smooth manifold $M$. 
\begin{definition} \label{metric}
The vector bundle $\mathbb{T}M:=TM\oplus T^\ast M$ can be endowed with a natural {\it indefinite inner product} on sections which has signature $(n,n)$ and is given by
\begin{equation}\label{innerproduct}
\langle X+\xi,Y+\eta \rangle =\dfrac{1}{2}(\xi(Y)+\eta(X)).
\end{equation}
\end{definition}

\begin{definition}\cite{H},\cite{G1}
	A {\it generalized almost complex structure} on $M$ is a map $\mathbb{J}: \mathbb{T}M\to\mathbb{T}M$ which defines an isometry of $(\mathbb{T}M,\langle \cdot,\cdot\rangle)$ and satisfies  $\mathbb{J}^2=-1$. It is said to be {\it integrable} if its $+i$-eigenbundle $\mathcal{L}\subset \mathbb{T}M\otimes \mathbb{C}$ is involutive with respect to the Courant bracket
	$$[X+\xi,Y+\eta]=[X,Y]+\mathcal{L}_X\eta -\mathcal{L}_Y\xi-\dfrac{1}{2}d(i_X\eta-i_Y\xi).$$
\end{definition}
It is simple to check that a maximal isotropic subbundle $\mathcal{L}$ of $\mathbb{T}M$ (or $\mathbb{T}M\otimes \mathbb{C}$) is Courant involutive if and only if $\textnormal{N}|\mathcal{L}=0$, where $\textnormal{N}$ denotes the Nijenhuis tensor
\begin{equation}\label{GeneralNij}
N(A,B,C)=\dfrac{1}{3}(\langle[A,B],C \rangle+\langle[B,C],A \rangle+\langle[C,A],B \rangle).
\end{equation}
It is important to notice that a manifold $M$ which admits a generalized almost complex structure is necessarily even dimensional.

Let $\mathbb{J}$ be a generalized almost complex structure on $M$. As $\mathbb{J}$ is an isometry with respect to the indefinite inner product \eqref{innerproduct}, the $+i$-eigenbundle $\mathcal{L}$ is an isotropic subbundle of $\mathbb{T}M\otimes \mathbb{C}$, and as such it can be expressed as the Clifford annihilator of a unique line subbundle $K_\mathcal{L}$ of the complex spinors for the metric bundle $\mathbb{T}M$. The fact that its annihilator is maximal isotropic implies that $K_\mathcal{L}\subset \bigwedge^\bullet T^\ast M\otimes \mathbb{C}$ is by definition a pure spinor line.

The bundle $\bigwedge^\bullet T^\ast M$ of differential forms can in fact be viewed as a spinor bundle for $\mathbb{T}M$, where the Clifford action of an element $X+\xi\in\mathbb{T}M$ on a differential form $\varphi$ is given by \eqref{CliffordAction}. It is simple to check that $(X+\xi)^2\varphi=\langle X+\xi,X+\xi\rangle\varphi$. Therefore, the canonical bundle $K_\mathcal{L}$ may be viewed as a smooth line subbundle of the complex differential forms 
satisfying the relation
$$\mathcal{L}=\{X+\xi\in \mathbb{T}M\otimes \mathbb{C}:\ (X+\xi)\cdot K_\mathcal{L}=0\}.$$
At every point $x\in M$, the line $K_\mathcal{L}$ is generated by a complex differential form $\varphi$ of special algebraic type which is determined by Proposition \ref{Chevalley}. Purity of the spinor $\varphi$ is equivalent to the fact that it has the form
\begin{equation}\label{PureSpinor}
\varphi=\exp(B+i\omega)\, \theta_1\wedge\cdots\wedge \theta_k,
\end{equation}
where  $B$, $\omega$ are the real and imaginary parts of a complex 2-form, and
$(\theta_1,\cdots,\theta_k)$ are linearly independent complex 1-forms spanning $\Ann (\Delta)$ where $\Delta=\pi_1(\mathcal{L})$. Here $\pi_1$ denotes the projection from $\mathbb{T}M\otimes\mathbb{C}$ onto $TM\otimes\mathbb{C}$; see \cite{C}.

The number $k$ is called  the {\it type} of $\mathbb{J}$ and it is not required to be constant along $M$. Points where the type is locally constant are called {\it regular}. If every point in $M$ is regular then $\mathbb{J}$ is called regular.
The type may take values from $k=0$ to $k=n$.

A generalized almost complex structure $\mathbb{J}$ on $M$ determines a decomposition of $\bigwedge^\bullet (T^\ast M\otimes\mathbb{C})$ 
into eigenspaces  of the Lie algebra action of $\mathbb{J}$, that is, the $ik$-eigenbundles of $\mathbb{J}$ acting in the spin representation; see \cite{G1}. These spaces can be concretely described as follows. For a given  $-n\leq k\leq n$, let 
$$U^{n-k}:=\bigwedge\!^k\overline{\mathcal{L}}\cdot K_{\mathcal{L}},$$
where $K_\mathcal{L}$ is the pure spinor line determined by $\mathbb{J}$ and $\cdot$ denotes the Clifford action of $\bigwedge^\bullet\overline{\mathcal{L}}$.

For the following examples we consider a $2$-form  $B$ on $M$,  acting by $B$-transformations whose definition \eqref{Btransform} extends naturally to manifolds. This is said to be a {\it $B$-field transformation} when $B$ is closed as well. 
Since identity \eqref{TypeNChange} implies that a $B$-transformation does not affect the projections to $TM\otimes \mathbb{C}$, 
it preserves type. For more details see \cite{G1}.

The following are some basic types of generalized complex structures, see \cite{G1}.
\begin{example}[Symplectic type $k=0$]\label{ExampleS}
	Given an {\it almost symplectic} manifold $(M,\omega)$ (that is, $\omega$ is a nondegenerate 2-form on $M$, not necessarily closed) we get that
	$$\mathbb{J}_\omega=\left( 
	\begin{array}{cc}
	0 & -\omega^{-1}\\
	\omega & 0
	\end{array}%
	\right),$$
	defines a generalized almost complex structure on $M$. Such structure is integrable if and only if $\omega$ is closed, in which case
	we say this is a generalized complex structure  of  {\it symplectic type}.
	
	We have that $\mathbb{J}_\omega$ determines a maximal isotropic subbundle $\mathcal{L}=\lbrace X-i\omega(X):\ X\in TM\otimes \mathbb{C}\rbrace$ and a pure spinor line generated by $\varphi_\mathcal{L}=\textnormal{exp}(i\omega)$. This generalized almost complex structure has type $k=0$, 
	constant through $M$. Thus, we may transform this example by a $B$-transformation and obtain another generalized almost complex structure of type $k=0$ as follows:
	$$e^{-B}\mathbb{J}_\omega e^B=\left( 
	\begin{array}{cc}
	1 & 0\\
	-B & 1
	\end{array}%
	\right)\left( 
	\begin{array}{cc}
	0 & -\omega^{-1}\\
	\omega & 0
	\end{array}%
	\right)\left( 
	\begin{array}{cc}
	1 & 0\\
	B & 1
	\end{array}%
	\right)=\left( 
	\begin{array}{cc}
	-\omega^{-1}B & -\omega^{-1}\\
	\omega+B\omega^{-1}B& B\omega^{-1}
	\end{array}%
	\right),$$
	$$e^{-B}(\mathcal{L})=\lbrace X-(B+i\omega)(X):\ X\in TM\otimes \mathbb{C}\rbrace,$$
	$$\varphi_{e^{-B}\mathcal{L}}=e^{B+i\omega}.$$
	By \eqref{PureSpinor} we see that any generalized complex structure of type $k=0$ is a $B$-transform of an almost symplectic structure.
	 Recall that if $\omega$ is closed, then $\mathbb{J}_\omega$ is integrable. It is simple to check that the generalized complex structure obtained 
	through $B$ in this way is also integrable if and only if $B$ is closed. We will refer to $e^{-B}\mathbb{J}_ce^B$ as a 
	{\it $B$-symplectic generalized complex structure}.
\end{example}

\begin{example}[Complex type $k=n$]\label{ExampleC}
	If $(M,J)$ is an almost complex manifold, we have that 
	$$\mathbb{J}_c=\left( 
	\begin{array}{cc}
	-J & 0\\
	0 & J^\ast
	\end{array}%
	\right),$$
	defines a generalized almost complex structure on $M$, which  is integrable if and only if $J$ is integrable in the classical sense.
	We refer to such a generalized complex structure as being of  {\it complex type}. This generalized almost complex structure has 
	type $k=n$, constant through $M$.
	
We have that $\mathbb{J}_c$ determines a maximal isotropic subbundle $\mathcal{L}=TM_{0,1}\oplus TM_{1,0}^\ast \subset \mathbb TM\otimes \mathbb C$ where $TM_{1,0}=\overline{TM_{0,1}}$ is the $+i$-eigenspace of $J$ and a pure spinor line generated by $\varphi_L=\Omega^{n,0}$ where $\Omega^{n,0}$ is any generator of the $(n,0)$-forms for the almost complex manifold $(M,J)$. This example may be transformed by a $B$-transformation
	$$e^{-B}\mathbb{J}_ce^B=\left( 
	\begin{array}{cc}
	-J & 0\\
	BJ+J^\ast B& J^\ast
	\end{array}%
	\right),$$
	$$e^{-B}(\mathcal{L})=\lbrace X+\xi-i_XB:\ X+\xi\in TM_{0,1}\oplus TM_{1,0}^\ast\rbrace,$$
	$$\varphi_{e^{-B}\mathcal{L}}=e^B\Omega^{n,0}.$$
	 \end{example}

\begin{remark}
Any generalized complex structure of type $k=n$ 
	is a $B$-transform of a complex structure, see \cite{G1}. 
Moreover, if $\mathbb{J}_c$ is integrable, then $e^{-B}\mathbb{J}_c e^B$ is 
	 integrable if and only if $B$ is closed.
\end{remark}

\begin{definition}\cite{G2}\label{ke}
	A {\it generalized almost K\"ahler structure} on $M$ is a pair $(\mathbb{J},\mathbb{J'})$ of commuting generalized almost complex structures 
	such that $-\mathbb{J}\mathbb{J'}$ is positive definite. 
	If moreover both $\mathbb{J}$ and $\mathbb{J'}$ are integrable we say that the pair $(\mathbb{J},\mathbb{J'})$ is a {\it generalized  K\"ahler structure}. 
	The corresponding  metric  on $\mathbb{T}M$  given 
	by
	$$G(X+\alpha, Y+\beta)=\langle \mathbb{J}(X+\alpha),\mathbb{J'}(Y+\beta)\rangle
	$$
	is called  a   {\it generalized K\"ahler metric}.
\end{definition}

\begin{example}\label{typical}
Assume that $(M,J,\omega)$ is a K\"ahler manifold satisfying the conditions of both Examples \ref{ExampleS} and \ref{ExampleC}.
This means that $J$ is a complex structure on $M$ and $\omega$ is a symplectic form of type $(1,1)$, satisfying $\omega J=-J^\ast \omega$,
which implies that $\mathbb{J}_c\mathbb{J}_\omega=\mathbb{J}_\omega\mathbb{J}_c$. It is simple to check that
$$G=-\mathbb{J}_c\mathbb{J}_\omega=\left(\begin{array}{cc}
0 & g^{-1}\\
g & 0
\end{array}%
\right),$$
defines a generalized K\"ahler  metric on $\mathbb{T} M$, 
therefore  $(\mathbb{J}_c,\mathbb{J}_\omega)$ is a generalized K\"ahler structure on $M$. 
Here $g$ denotes the Riemannian metric on $M$ induced by the formula $g(X,Y)=\omega(X,JY)$.
\end{example}

\begin{remark}\label{SignatureKahler}\cite{G2}
We may regard the generalized K\"ahler metric $G$ of Example \ref{typical} as the standard example.
In fact, given a generalized K\"ahler structure $(\mathbb{J},\mathbb{J'})$ on $M$
with generalized metric $G$ there always exist a Riemannian metric $g$ on $M$ and a $2$-form $b\in \Omega^2(M)$ such that
$$G=e^{b}\left( 
\begin{array}{cc}
 & g^{-1}\\
 g &   
\end{array}%
\right)e^{-b}.$$
Thus, the signature of any generalized K\"ahler metric is $(n,n)$.
\end{remark}

\section{Invariant generalized complex structures}\label{igcs}
We are interested in exploring the properties of the invariant generalized complex structures on maximal flags found by the third author and San Martin in \cite{VS}. With this in mind we set up some terminology.

Let $G$ be a semisimple Lie group with Lie algebra $\mathfrak{g}$ and $\mathfrak{h}\subset\mathfrak{g}$ be a Cartan subalgebra of $\mathfrak{g}$. If $\Pi$ denotes the set of roots of the pair $(\mathfrak{g},\mathfrak{h})$ we can write
$$\mathfrak{g}=\mathfrak{h}\oplus \sum_{\alpha \in \Pi}\mathfrak{g}_\alpha,$$
where $\mathfrak{g}_\alpha=\lbrace X\in\mathfrak{g}:\ [H,X]=\alpha(H)X,\ \forall H\in\mathfrak{h}\rbrace$ denotes the respective complex 1-dimensional root space.

We denote by $\langle \cdot,\cdot \rangle$ the Cartan--Killing form on $\mathfrak{g}$ and fix a Weyl basis of $\mathfrak{g}$ which amounts to taking
 $X_\alpha \in\mathfrak{g}_\alpha$ such that $\langle X_\alpha, X_{-\alpha}\rangle=1$ and $[X_\alpha,X_\beta]=m_{\alpha,\beta}X_{\alpha+\beta}$ with $m_{\alpha,\beta}\in\mathbb{R}$, $m_{-\alpha,-\beta}=-m_{\alpha,\beta}$, and $m_{\alpha,\beta}=0$ if $\alpha+\beta$ is not a root.

Let $\Pi^+$ be a choice of positive roots and $\Sigma$ be the corresponding set of simple roots. The Borel subalgebra corresponding to  $\Pi^+$ is $\mathfrak{b}=\mathfrak{h}\oplus\sum_{\alpha\in\Pi^+ }\mathfrak{g}_\alpha$ and we use the notation
$$\mathfrak{n}^-= \sum_{\alpha\in\Pi^+ }\mathfrak{g}_{-\alpha} \qquad\textnormal{and}\qquad \mathfrak{n}^+= \sum_{\alpha\in\Pi^+ }\mathfrak{g}_\alpha.$$
A subalgebra $\mathfrak{p}$ is parabolic if $\mathfrak{b}\subset \mathfrak{p}$. If $P$ denotes the normalizer of $\mathfrak{p}$ in $G$ then the homogeneous space $\mathbb{F}=G/P$ will be called a {\it flag manifold}.

\begin{definition}
A {\it maximal flag} $\mathbb{F}=G/P$ is a flag manifold such that $\mathfrak{b}=\mathfrak{p}$.
\end{definition}

Every maximal flag manifold can also be seen as the quotient $\mathbb{F}=U/T$ where $U$ is a compact real form in $G$ and $T$ is a maximal torus in $U$.

\begin{example}\label{CaseSln} 
When $\mathfrak{g}=\mathfrak{sl}(n,\mathbb{C})$ we have the following data:
\begin{enumerate}
	\item[$\iota$.] The Cartan--Killing form $\langle X,Y \rangle$ is equal to $2n$ times the form $\textnormal{tr}(XY)$.
	\item[$\iota\iota$.] A canonical choice for a Cartan subalgebra $\mathfrak{h}$ is the set of diagonal matrices.
	
	\item[$\iota\iota\iota$.] With this choice of $\mathfrak{h}$, the roots are the linear functionals 
	$\alpha_{ij}(\textnormal{diag}(\lambda_1,\cdots,\lambda_{n}))=\lambda_i-\lambda_j$ for $i\neq j$.
It is simple to check that $\mathfrak{g}_{\alpha_{ij}}$ are the subspaces generated by the elements given by the matrices $E_{ij}$ (with $1$ in the $i,j$ entry and zeros elsewhere).
	
	\item[$\iota \nu$.] The compact real form of $\mathfrak{sl}(n,\mathbb{C})$ is $\mathfrak{su}(n)$, the real Lie algebra of anti-Hermitian matrices.
	
	\item[$\nu$.] $H\in\mathfrak{h}$ is regular if and only if its eigenvalues are all distinct.
	
	\item[$\nu\iota$.] {\it Flag manifolds:} Let $\mathbb{F}(n_1,\cdots,n_s)$ be the set of all increasing sequences 
	$\lbrace V_1\subset \cdots \subset V_s \rbrace$
	of complex vector subspaces of $\mathbb{C}^{n}$ with $\dim V_j=n_j$ and 
	$n_1<n_2<\cdots<n_s=n$. The set $\mathbb{F}(n_1,\cdots,n_s)$ has the structure of a homogeneous space
	$$\mathbb{F}(n_1,\cdots,n_s)\approx \dfrac{\textnormal{SU}(n)}{S(\textnormal{U}(d_1)\times \cdots \times \textnormal{U}(d_s))}\approx \textnormal{Ad}(\textnormal{SU}(n))(H),$$
	where $d_i=n_i-n_{i-1}$ with $n_0=0$, the element $H=\textnormal{diag}(i\lambda_1I_{d_1},\cdots,i\lambda_sI_{d_s}) \in \mathfrak{t}$ with $\lambda_i\neq \lambda_j$ for $i\neq j$ and $d_1\lambda_1+\cdots+d_s\lambda_s=0$ . 
	
	All flag manifolds of $\textnormal{SU}(n)$ are given by such quotients, and are determined by the choices
	of  integers $n_i$. These are also all the flag manifolds of $\textnormal{SL}(n)$.
	
	In the particular case when $n_i=i$ for all $1\leq i \leq n-1$, we then have that 
	$T= S(\textnormal{U}(1)\times \cdots \times \textnormal{U}(1))$ ($n$ times) is a maximal torus and $\textnormal{Lie}(T)=\mathfrak{t}$.
	Then the quotient  has structure of homogeneous space
	$$\mathbb{F}(1, 2, \cdots,n-1)\approx \dfrac{\textnormal{SU}(n)}{S(\textnormal{U}(1)\times
		\cdots \times \textnormal{U}(1))}\approx \textnormal{Ad}(\textnormal{SU}(n))(H),$$
	where $H=\textnormal{diag}(i\lambda_1,\cdots,i\lambda_{n}) \in \mathfrak{t}$ with $\lambda_i\neq \lambda_j$ for $i\neq j$ and $\lambda_1+\cdots+\lambda_{n}=0$. 
	The flags $\mathbb{F}(1, 2, \cdots,n-1)$ are called {\it maximal flags} or {\it full flags} and in this work we concerned ourselves only with those.

\end{enumerate}

\end{example}

From now on, we will consider only maximal flag manifolds. Accordingly, using the Cartan--Killing form we may identify the tangent and   cotangent spaces of $\mathbb{F}$ at the origin $b_0$ with  $T_{b_0}\mathbb{F}=\mathfrak{n}^-$ and  with $T_{b_0}^\ast\mathbb{F}=\mathfrak{n}^+$, respectively.

Since the Cartan--Killing form restricted to $\mathfrak{h}$ is nondegenerate, for every root $\alpha\in \mathfrak{h}^\ast$ 
 there exists a unique $H_\alpha\in\mathfrak{h}$ such that $\alpha=\langle H_\alpha,\cdot\rangle$. The real vector space generated by the $H_\alpha's$ with $\alpha\in \mathfrak{h}^\ast$ will be denoted by $\mathfrak{h}_\mathbb{R}$. Let $\mathfrak{u}$ be the compact real form of $\mathfrak{g}$ which is given by
$$\mathfrak{u}=\mathbb{R}\lbrace i\mathfrak{h}_\mathbb{R},A_\alpha,S_\alpha:\ \alpha\in\Pi^+ \rbrace,$$
where $A_\alpha=X_\alpha-X_{-\alpha}$ and $S_\alpha=i(X_\alpha+X_{-\alpha})$. If $U=\textnormal{exp}(\mathfrak{u})$ is the corresponding compact real form in $G$, then we may write $\mathbb{F}=U/T$ where $T=U\cap P$ is a maximal torus in $U$. 
Alternatively,  $\mathbb{F}=\textnormal{Ad}(U)(H)$ the adjoint orbit of a regular element $H$ of the Lie algebra $\mathfrak{t}$ of $T$ which remains fixed throughout the paper. In the latter case we are identifying $H$ with the origin $b_0$. Under this identification $T_{b_0}\mathbb{F}=\sum_{\alpha\in\Pi^+ }\mathfrak{u}_\alpha$ where $\mathfrak{u}_\alpha=(\mathfrak{g}_{-\alpha}\oplus \mathfrak{g}_{\alpha})\cap\mathfrak{u}$.

It is also simple to see that the adjoint representation of $U$ induces a natural action on $\mathbb{F}$. The authors in \cite{VS} determined those generalized almost complex structures on $\mathbb{F}$ which are $U$-invariant. Because 
 $U$-invariant geometric structures on $\mathbb{F}$ are completely characterized by their value at the origin $b_0$ of $\mathbb{F}$, it was sufficient to find all maps $\mathbb{J}:\mathfrak{n}^-\oplus(\mathfrak{n}^-)^\ast\to \mathfrak{n}^-\oplus(\mathfrak{n}^-)^\ast$ that are isometries with respect to $\langle\cdot,\cdot\rangle$, satisfy $\mathbb{J}^2=-1$, and commute with  the adjoint action of $T$ on $\mathfrak{n}^-\oplus(\mathfrak{n}^-)^\ast$. Note that in this case the indefinite inner product given in \eqref{innerproduct} is just the Cartan--Killing form.

Aiming to describe the form of these invariant generalized almost complex structures on $\mathbb{F}$, we identify $\mathfrak{u}\cong \mathfrak{u}^\ast$ and $\mathfrak{u}_\alpha\cong \mathfrak{u}_\alpha^\ast$ by means of the Kirillov--Kostant--Souriau (KKS) symplectic form on $\textnormal{Ad}(U)(H)$ which at $b_0$ is given by 
$$\omega_{b_0}(\widetilde{X},\widetilde{Y})=\langle H,[X,Y]\rangle,$$
for all $X,Y\in\mathfrak{u}$ and where $\widetilde{X}=\textnormal{ad}(X)$ denotes the fundamental vector field associated to $X$ with respect to the adjoint action. The elements of $\mathfrak{u}_\alpha^\ast$ will be denoted by $A_\alpha^\ast$ and $S_\alpha^\ast$ where
$$X^\ast=\dfrac{1}{2i\langle H,H_\alpha\rangle}\omega(\widetilde{X},\cdot).$$ 

With this notation $\lbrace -S_\alpha^\ast, A_\alpha^\ast\rbrace$ is a dual basis for $\lbrace A_\alpha, S_\alpha\rbrace$.
As a first consequence of these identifications we get that for triples of roots $\alpha,\beta,\gamma\in\Pi^+$, the Nijenhuis tensor \eqref{GeneralNij} can be rewritten as
\begin{equation}\label{ParticularNij}
N(A,B,C)=\dfrac{1}{12}(k_\gamma\langle H,[C_2,[A_1,B_1]] \rangle+k_\alpha\langle H,[A_2,[B_1,C_1]] \rangle+k_\beta\langle H,[B_2,[C_1,A_1]] \rangle),
\end{equation}
where $A=A_1+A_2^\ast\in \mathfrak{u}_\alpha\oplus \mathfrak{u}_\alpha^\ast$, $B=B_1+B_2^\ast\in\mathfrak{u}_\beta\oplus\mathfrak{u}_\beta^\ast$, $C=C_1+C_2^\ast\in\mathfrak{u}_\gamma\oplus \mathfrak{u}_\gamma^\ast$, and $k_\alpha=\dfrac{1}{2i\langle H,H_\alpha\rangle}$ for all $\alpha\in \Pi^+$.

\begin{remark}\label{canonical}\cite{VS} Invariant generalized almost complex structures $\mathbb{J}$ on $\mathbb{F}$ are those whose restriction to $\mathfrak{u}_\alpha\oplus\mathfrak{u}_\alpha^\ast$ have the form either
$$\pm \mathcal{J}_0=\pm \left( 
\begin{array}{cccc}
0 & -1 & 0 & 0\\ 
1 & 0 & 0 & 0\\
0 & 0 & 0 & -1\\
0 & 0 & 1 & 0
\end{array}%
\right):=\pm\left( 
\begin{array}{cc}
-J_0 & 0\\
0 & J^\ast_0
\end{array}%
\right)\qquad \textbf{complex type},$$
or else
$$\mathcal{J}_\alpha=\left( 
\begin{array}{cccc}
a_\alpha & 0 & 0 & -x_\alpha\\ 
0 &  a_\alpha& x_\alpha & 0\\
0 & -y_\alpha & -a_\alpha & 0\\
y_\alpha & 0 & 0 & -a_\alpha
\end{array}%
\right):=\left( 
\begin{array}{cc}
\mathcal{A}_{\alpha} & \mathcal{X}_{\alpha} \\
\mathcal{Y}_{\alpha}  & -\mathcal{A}_{\alpha}
\end{array}%
\right)\qquad \textbf{noncomplex type},$$
with $a_\alpha,x_\alpha,y_\alpha\in\mathbb{R}$ such that $a_\alpha^2=x_\alpha y_\alpha-1$.
\end{remark}

Thus, 
we obtain a concrete description of  all  invariant generalized almost complex structures on $\mathbb F$ as follows.
 Consider the set $\mathcal M_a$ of all invariant generalized almost complex structures on the maximal  flag $\mathbb F$. For a fixed root $\alpha \in \Pi^+$,
let $\mathcal M_\alpha$ be the  restriction of $\mathcal M_a$ to the subspace $\mathfrak u_\alpha\oplus \mathfrak u_\alpha^*$. 

\begin{lemma} \label{Aalpha}
  $\mathcal M_\alpha$ is a disjoint union of:
\begin{enumerate}
\item[$\iota$.] a 2-dimensional family of structures of noncomplex type, parametrized by the real algebraic surface
cut out by $a_\alpha^2-x_\alpha y_\alpha =-1$ in $\mathbb R^3$, and 
\item[$\iota\iota$.]  2 extra points, corresponding to the complex structures $\pm \mathcal{J}_0$. 
\end{enumerate}
\end{lemma}

\begin{remark}\label{topoS}
Note that while  points appearing in item $\iota.$  satisfy  $a_\alpha^2-x_\alpha y_\alpha =-1$,  points appearing
in item $\iota\iota.$ correspond to matrices that have entries 11, 14, and 41 all zero, so that the same polynomial expression 
would give  $a_\alpha^2-x_\alpha y_\alpha =0$. Thus, the 2 extra points do not belong to the closure of 
the given surface. 
Alternatively, we can also see that, 
if $\mathcal M_\alpha$ is endowed with the induced topology from 4x4 matrices (or $\mathbb{R}^{16}$), 
it is not possible to construct a sequence of points  appearing in item $\iota.$ which converges 
to a  point from item $\iota\iota.$
 Hence,
 these 2 extra points are isolated points  with either the natural topology of $\mathcal M_\alpha$ 
 induced from 4x4 matrices, or else from the topology induced by the embedding in $\mathbb R^3$.
\end{remark}

\begin{notation} For each $\alpha\in \Pi^+$ we will denote by $\mathbb{J}_\alpha$ the restriction of an invariant generalized almost complex structure $\mathbb{J}$ to $\mathfrak{u}_\alpha\oplus\mathfrak{u}_\alpha^\ast$.
\end{notation}

The maximal invariant isotropic $+i$-eigenspace $\mathcal{L}$ of $(\mathfrak{n}^-\oplus (\mathfrak{n}^-)^\ast)\otimes \mathbb{C}$ associated to an invariant generalized almost complex structure $\mathbb{J}$ can be written as a sum $\mathcal{L}=\bigoplus_{\alpha\in\Pi^+}\mathcal{L}_\alpha$, where
\begin{equation}\label{isotropic1}
\mathcal{L}_\alpha=\left\lbrace \begin{array}{lcc} 
L_0^{\pm} = \mathbb{C} \{ A_\alpha \mp iS_\alpha , A^\ast _\alpha \mp S^\ast _\alpha \} & \textnormal{if} & \mathbb{J}_\alpha=\pm \mathcal{J}_0\\
L_\alpha = \mathbb{C} \{ x_\alpha A_\alpha + (a_\alpha - i)A^\ast _\alpha , x_\alpha S_\alpha + (a_\alpha - i)S^\ast _\alpha \} & \textnormal{if} & \mathbb{J}_\alpha= \mathcal{J}_\alpha.
\end{array}
\right.
\end{equation}
Recall that a generalized 
almost complex structure $\mathbb{J}$ is integrable when the Nijenhuis tensor restricted to the $+i$-eigenbundle of $\mathbb{J}$ vanishes. 
Integrability 
depends just on  what happens with triples of the form $(\mathbb{J}_\alpha,\mathbb{J}_\beta,\mathbb{J}_{\alpha+\beta})$
associated to positive roots $(\alpha,\beta,\alpha+\beta)$. This is so, because for the case of triples of the form $(\alpha,\beta,\gamma)$ such that $\gamma\neq \alpha+\beta$ the Nijenhuis tensor vanish automatically. Accordingly, we obtain that $\mathbb{J}$
 is integrable if and only  if for all positive roots  which triple $(\mathbb{J}_\alpha,\mathbb{J}_\beta,\mathbb{J}_{\alpha+\beta})$ 
 corresponds to one of the rows appearing in Table \ref{integrab}
\begin{table}
	\begin{tabular}{c|rrr}
		$\pm$ & $\mathbb{J}_\alpha$ & $\mathbb{J}_\beta$ & $\mathbb{J}_{\alpha + \beta}$ \\
		\hline
		 & $\mathcal{J}_0$ & $ \mathcal{J}_0$ & $ \mathcal{J}_0$\\
		 & $  \mathcal{J}_0$ & $ - \mathcal{J}_0$ & $  \mathcal{J}_0$\\
		& $  \mathcal{J}_0$ & $ - \mathcal{J}_0$ & $  -  \mathcal{J}_0$\\
		 & $ \mathcal{J}_\alpha$ & $ \mathcal{J}_0 $ & $  \mathcal{J}_0$ \\
		 & $  \mathcal{J}_0$ & $\mathcal{J}_\beta $ & $  \mathcal{J}_0$ \\
		 & $ \mathcal{J}_0$ & $ - \mathcal{J}_0$ & $\mathcal{J}_{\alpha+\beta} $\\
		 & $\mathcal{J}_\alpha$ & $\mathcal{J}_\beta$ & $\mathcal{J}_{\alpha+\beta} $
	\end{tabular}
\caption{Integrability conditions}\label{integrab}
\end{table}
where,  for the  last row, additional  conditions (\ref{Integrabilityconditions}) are required:
\begin{equation}\label{Integrabilityconditions}
\left\lbrace \begin{array}{cc}
a_{\alpha+\beta} x_\alpha x_\beta - a_\beta x_\alpha x_{\alpha+\beta} - a_\alpha x_\beta x_{\alpha+\beta} = 0 \\
x_\alpha x_\beta - x_\alpha x_{\alpha+\beta} - x_\beta x_{\alpha+\beta} = 0.
\end{array}\right.
\end{equation}
The $\pm$ sign on top of Table \ref{integrab} means that each row may be taken either with plus or with minus signs, for instance both triples
$(\mathcal{J}_0, \mathcal{J}_0, \mathcal{J}_0)$ and $(-\mathcal{J}_0,- \mathcal{J}_0, -\mathcal{J}_0)$ are integrable. 
%
%

These statements were obtained by 
direct computation of the 
Nijenhuis tensor $N|_{\mathcal{L}_{\alpha\beta}}$ where $\mathcal{L}_{\alpha\beta}=\mathcal{L}_{\alpha}\cup\mathcal{L}_{\beta}\cup \mathcal{L}_{\alpha+\beta}$,
taking into account  that for each $\alpha\in\Pi^+$ the space  $\mathcal{L}_\alpha$ can be either $L_0^{\pm}$ or $L_\alpha$. For more details see \cite{VS}.

If $\mathbb{J}$ is an invariant generalized complex structure, the set
$\Pi^+_J$ composed by the elements $\alpha\in \Pi$ such that $\mathbb{J}_\alpha$ is of complex type with $\mathbb{J}_\alpha=\mathcal{J}_0$ together with the elements $\alpha\in \Pi$ such that $\mathbb{J}_\alpha$ is of noncomplex type with $x_\alpha>0$, is a choice of positive roots with respect to some lexicographic order in $\mathfrak{h}^\ast_\mathbb{R}$. We will denote by $\Sigma_J$ the simple root system of $\mathfrak{g}$ associated to $\Pi^+_J$. 

We also denote by $\langle \Theta \rangle$  the smallest  subset of $\Sigma$ containing $\Theta$ which is closed 
for addition, thus,  if $\alpha,\beta \in \langle \Theta\rangle$ and $\alpha+\beta$ is a root, then 
$\alpha+\beta \in \langle \Theta \rangle$. Furthermore, we denote $\langle \Theta \rangle^+:= \langle \Theta \rangle\cap\Pi^+$. 

We will use the following results:

\begin{theorem}\cite{VS}\label{theta}
If $\mathbb{J}$ is an invariant generalized complex structure on $\mathbb{F}$, then there exists a subset $\Theta \subseteq \Sigma_J$,  such that $\mathbb{J}_\alpha$ is of noncomplex type for each $\alpha \in \langle \Theta \rangle^+$ and of complex type for each $\alpha \in \Pi^+ \backslash \langle \Theta \rangle^+$.
\end{theorem}
Conversely,
\begin{theorem}\cite{VS}\label{theta2}
Let $\Sigma$ be a simple root system for $\mathfrak g$ and consider $\Theta\subseteq \Sigma$. Then there exists an invariant generalized complex structure $\mathbb{J}$ on $\mathbb{F}$ such that $\mathbb{J}_\alpha$ is of noncomplex type for each $\alpha \in \langle \Theta \rangle^+$ and of complex type for each $\alpha \in \Pi^+ \backslash \langle \Theta \rangle^+$.
\end{theorem}

\section{Effects of the action by the Weyl group}\label{S:4}
Similarly to what happens in the classical case of invariant almost complex structures on flag manifolds (see \cite{SN}), we can describe the effects of the action by the Weyl group on the set of invariant generalized almost complex structures on $\mathbb{F}$.

Let $\mathcal{W}$ be the Weyl group generated by reflections with respect to the roots $\alpha\in \Pi$. It is well known that action of $\mathcal{W}$ on $\mathfrak{h}^\ast$ leaves $\Pi$ invariant and $\mathcal W \simeq N_U(\mathfrak{h})/T$ where $N_U(\mathfrak{h})$ represents the normalizer of $\mathfrak{h}$ in $U$ with respect the adjoint action. The group $N_U(\mathfrak{h})$ acts on $\mathfrak{n}^-\otimes \mathbb{C}$, and consequently on $(\mathfrak{n}^-\oplus (\mathfrak{n}^-)^\ast)\otimes \mathbb{C}$, by permuting the root systems. Thus, if $\mathbb{J}:\mathfrak{n}^-\oplus(\mathfrak{n}^-)^\ast\to \mathfrak{n}^-\oplus(\mathfrak{n}^-)^\ast$ is an invariant generalized almost complex structure on $\mathbb{F}$ then
\begin{equation}\label{WeylAction}
\overline{w}\cdot\mathbb{J}\cdot \overline{w}^{-1}:= (\Ad\oplus \Ad^\ast)(w)\circ \mathbb{J}\circ (\Ad\oplus \Ad^\ast)(w^{-1})
\end{equation}
is also an invariant generalized almost complex structure of the same type as $\mathbb{J}$, where  $w$ is any representative of the class 
$\overline{w}$ in $N_U(\mathfrak{h})$. Here we are using the fact that $\Ad(w)$ is an isometry with respect to the Cartan--Killing form. Since $\overline{w}\cdot\mathbb{J}\cdot\overline{w}^{-1}$ does not pend on the choice of representative  $\overline{w}$, we may use it for inducing a well defined action of the Weyl group $\mathcal{W}$ on the set of invariant generalized almost complex structures on $\mathbb{F}$.  Now, note  that $\mathbb{J}$ and $\overline{w}\cdot\mathbb{J}\cdot\overline{w}^{-1}$ are isomorphic as almost complex structures, that is, because the following diagram commutes
$$\xymatrix{
	\mathfrak{n}^-\oplus (\mathfrak{n}^-)^\ast \ar[d]_{\mathbb{J}}\ar[r]^{f(\overline{w})} & \mathfrak{n}^-\oplus (\mathfrak{n}^-)^\ast \ar[d]^{\overline{w}\cdot\mathbb{J}\cdot\overline{w}^{-1}}\\
	\mathfrak{n}^-\oplus (\mathfrak{n}^-)^\ast \ar[r]_{f(\overline{w})} & \mathfrak{n}^-\oplus (\mathfrak{n}^-)^\ast
}$$
where $f(\overline{w}):=(\Ad\oplus \Ad^\ast)(w)$.  We will denote the action given in \eqref{WeylAction} by $w\cdot \mathbb{J}$.	
\begin{remark}
	When we say that $N_U(\mathfrak{h})$ leaves invariant the set of roots associated to $\mathfrak{h}$ by permuting it, we mean that for $\alpha \in \Pi$ we get that $w\cdot \alpha:=\alpha\circ \Ad(w)\in \Pi$. This is because $\Ad(w)$ is an automorphism. It is easy to check that the root space associated to the root $\alpha\circ \Ad(w)$ is $\mathfrak{g}_{w\cdot \alpha}= \Ad(w^{-1})\mathfrak{g}_\alpha$. Therefore, we have that 
	$$\mathfrak{u}_{w\cdot \alpha}\oplus \mathfrak{u}_{w\cdot \alpha}^\ast=(\Ad\oplus \Ad^\ast)(w^{-1})(\mathfrak{u}_\alpha\oplus\mathfrak{u}^\ast_\alpha).$$
\end{remark}
If $\mathbb{J}=\bigoplus_{\alpha\in\Pi^+}\mathbb{J}_\alpha$ we can get an explicit description of the action \eqref{WeylAction}. Given  $w\in \mathcal{W}$ we have that
\begin{eqnarray*}
	w\cdot \mathbb{J} & = & \bigoplus_{\alpha\in\Pi^+}w\cdot\mathbb{J}_\alpha\\
	& = & \bigoplus_{\alpha\in\Pi^+}(\Ad\oplus \Ad^\ast)(w)\circ \mathbb{J}_\alpha\circ (\Ad\oplus \Ad^\ast)(w^{-1})\\
	& = & \bigoplus_{\alpha\in\Pi^+}\mathbb{J}_{w^{-1}\cdot\alpha},
\end{eqnarray*}
where $w^{-1}\cdot\alpha=\alpha\circ \Ad(w^{-1})$.

It is well known that if $\Sigma$ is a simple root system of $\mathfrak{g}$ then for each $w\in \mathcal{W}$ the set $w\cdot \Sigma:=\lbrace \alpha\circ \Ad(w^{-1}):\ \alpha\in\Sigma\rbrace$ defines another simple root system of $\mathfrak{g}$. We will denote by $w\cdot \Pi^+$ the corresponding set of positive roots associated to $w\cdot \Sigma$. It is simple to check that every triple of roots $(\alpha,\beta,\alpha+\beta)$ must be sent, by the action of the Weyl group, into one of the following possibilities: $(\alpha,\beta,\alpha+\beta)$, $(-\beta,\alpha+\beta,\alpha)$, $(-\alpha,\alpha+\beta,\beta)$, $(\alpha,-\alpha-\beta,-\beta)$, $(\beta,-\alpha-\beta,-\alpha)$, $(-\alpha,-\beta,-\alpha-\beta)$.
  
\begin{proposition}\label{WeylIntegrability}
	Let $\Sigma$ be a simple root system for $\mathfrak{g}$ with $\Pi^+$ the corresponding set of positive roots. If $\mathbb{J}$ is an invariant generalized complex structure on $\mathbb{F}$, then $w\cdot \mathbb{J}$ defines another invariant generalized complex structure on $\mathbb{F}$ of the same type.
\end{proposition}
\begin{proof}
	We know that $\mathbb{J}$
	is integrable if and only if for each triple of positive roots $(\alpha,\beta,\alpha+\beta)$ we have that $(\mathbb{J}_\alpha,\mathbb{J}_\beta,\mathbb{J}_{\alpha+\beta})$ 
	corresponds to one of the rows of Table \ref{integrab}. If $\mathbb{J}_\alpha=\pm\mathcal{J}_0$ is of complex type then $\mathbb{J}_{-\alpha}=\mp\mathcal{J}_0$ and if $\mathbb{J}_\alpha=\mathcal{J}_\alpha$ is of noncomplex type with
	$\mathcal{J}_\alpha={\tiny \left( 
	\begin{array}{cccc}
	a_\alpha & 0 & 0 & -x_\alpha\\ 
	0 &  a_\alpha& x_\alpha & 0\\
	0 & -y_\alpha & -a_\alpha & 0\\
	y_\alpha & 0 & 0 & -a_\alpha
	\end{array}%
	\right)}$ then $\mathbb{J}_{-\alpha}$ is also of noncomplex type where $a_{-\alpha}=a_\alpha$, $x_{-\alpha}=-x_\alpha$, and $y_{-\alpha}=-y_\alpha$ (see \cite{VS}). Therefore, a straightforward computation allows us to show that $(\mathbb{J}_{\gamma_1},\mathbb{J}_{\gamma_2},\mathbb{J}_{\gamma_3})$, where $(\gamma_1,\gamma_2,\gamma_3)$ is one of the six possibilities given above, corresponds to one of the rows of Table \ref{integrab}. It is important to notice that the identities $a_{-\alpha}=a_\alpha$ and $x_{-\alpha}=-x_\alpha$ imply that the systems obtained when $\mathbb{J}_{\gamma_1}$, $\mathbb{J}_{\gamma_2}$ and $\mathbb{J}_{\gamma_3}$ are of noncomplex type are equivalent to system \eqref{Integrabilityconditions}.
\end{proof}

\begin{remark}\label{x>0}
	As consequence of Theorem \ref{theta} and Proposition \ref{WeylIntegrability} we have that, up to the action of the Weyl group  an invariant generalized complex structure $\mathbb{J}$ on $\mathbb{F}$ satisfies (see also \cite{VS}):
	\begin{enumerate}
		\item[$\iota.$] if $\mathbb{J}_\alpha$ is of noncomplex type then $x_\alpha>0$, or
		\item[$\iota\iota.$] if $\mathbb{J}_\alpha$ is of complex type then $\mathbb{J}_\alpha=\mathcal{J}_0$.
	\end{enumerate}
\end{remark}

\section{Moduli space of generalized almost complex structures}\label{S:5}
The aim of this section is to classify all invariant generalized almost complex structures on a maximal flag $\mathbb{F}$ up to invariant $B$-transformations. As a consequence 
 we will then give an explicit expression for the invariant pure spinor line associated to each of
these structures. We will see that when we restrict the set of all invariant generalized almost complex structures to a specific subspace
 $\mathfrak{u}_\alpha\oplus \mathfrak{u}_\alpha^\ast$, then using $B$-transformations, all  structures of noncomplex type can be obtained from  structures of symplectic type. Hence, modulo 
$B$-transformations there remain only structures of either complex or symplectic types on  $\mathfrak{u}_\alpha$.

Now we explore the concept of moduli space of generalized almost complex structures using $B$-transformations. Let $M$ be a smooth manifold and consider the set
$$\mathcal{B}:=\lbrace e^B:\ B\in \Omega^2(M)\rbrace.$$
It is simple to check that $\mathcal{B}$ is a group with the natural commutative product
$$e^{B_1}\cdot e^{B_2}=\left( 
\begin{array}{cc}
1 & 0\\
B_1 & 1
\end{array}%
\right)\left( 
\begin{array}{cc}
1 & 0\\
B_2 & 1
\end{array}%
\right)=e^{B_1+B_2}.$$
If $\Omega_{cl}^2(M)$ denotes the set of closed 2-forms, then $\mathcal{B}_c=\lbrace e^B:\ B\in \Omega_{cl}^2(M)\rbrace$ is a subgroup of $\mathcal{B}$.
Let $\mathcal{M}_a$ (resp. $\mathcal{M}$) be the set of all generalized almost complex structures (resp. generalized complex structures) on $M$. It is easy to see that the map $\mathcal{B}\times \mathcal{M}_a\to \mathcal{M}_a$ given by $e^B\cdot \mathbb{J}:=e^{-B}\mathbb{J}e^B$ is a well-defined action of $\mathcal{B}$ on $\mathcal{M}_a$. The same expression allows us to give a well-defined action of $\mathcal{B}_c$ on $\mathcal{M}$.

We will deal with the following concept:
\begin{definition}\label{Bmoduli}
The moduli space of generalized almost complex structures on $M$ is defined as the quotient space $\mathfrak{M}_a:=\dfrac{\mathcal{M}_a}{\mathcal{B}}$ which is determined by the action above. The moduli space of generalized complex structures on $M$ is defined in a similar way as $\mathfrak{M}:=\dfrac{\mathcal{M}}{\mathcal{B}_c}$.
\end{definition}
\begin{remark}\label{int}
Every orbit in $\mathfrak{M}_a$ is composed of generalized almost complex structures of the same type.
\end{remark}
We will describe these moduli spaces in the case of invariant generalized complex structures on a maximal flag manifold $\mathbb{F}$. For this we consider differential 2-forms that are invariant by the adjoint representation.
\subsection{Effects of the action by invariant $B$-transformations}\label{SB:5.1}
Consider an invariant generalized almost complex structure $\mathbb{J}$ on a maximal flag $\mathbb{F}$ and  let $\alpha$ be a positive root.
\begin{remark}\label{SymplecticInvariantType}
	Suppose that $\mathbb{J}_\alpha$ is of noncomplex type with $a_\alpha=0$, then $\mathbb{J}_\alpha$ is of symplectic type. Indeed, the condition $x_\alpha y_\alpha=1$ and the identity $[A_\alpha,S_\alpha]=2iH_\alpha$ imply that
	$$\mathbb{J}_\alpha=\left( 
	\begin{array}{cc}
	0 & -\omega_\alpha^{-1}\\
	\omega_\alpha & 0
	\end{array}%
	\right),$$
	where $\omega_\alpha=\left( 
	\begin{array}{cc}
	0 & -1/x_\alpha \\
	1/x_\alpha  & 0
	\end{array}%
	\right)=-\dfrac{k_\alpha}{x_\alpha}\cdot \omega_{b_0}|_{\mathfrak{u}_\alpha}$. Here $\omega_{b_0}$ denotes the KKS symplectic form on $\mathfrak{n}^-$ and $k_\alpha=\dfrac{1}{2i\langle H,H_\alpha\rangle}$. 
\end{remark}
 Recall that the maximal isotropic subspace $\mathcal{L}_\alpha$ associated to $\mathbb{J}_\alpha$ is determined by \eqref{isotropic1}. If $\pi_\alpha$ denotes the projection from $\mathfrak{u}_\alpha\oplus \mathfrak{u}_\alpha^\ast$ onto $\mathfrak{u}_\alpha$, then $\Delta_\alpha=\pi_\alpha (\mathcal{L}_\alpha)$ is

$$\Delta_\alpha=\left\lbrace \begin{array}{lcc} 
 \mathbb{C} \{ A_\alpha \mp iS_\alpha\} & \textnormal{if} & \mathbb{J}_\alpha=\pm \mathcal{J}_0\\
 \mathbb{C} \{ x_\alpha A_\alpha, x_\alpha S_\alpha \} & \textnormal{if} & \mathbb{J}_\alpha= \mathcal{J}_\alpha.
\end{array}
\right.$$ 
Moreover
$$\Ann \Delta_\alpha=\left\lbrace \begin{array}{lcc} 
\mathbb{C} \{ A_\alpha^\ast \mp iS_\alpha^\ast\} & \textnormal{if} & \mathbb{J}_\alpha=\pm \mathcal{J}_0\\
 \{ 0 \} & \textnormal{if} & \mathbb{J}_\alpha= \mathcal{J}_\alpha
\end{array}
\right.,$$ 
since $\Ann \Delta_\alpha=\mathcal{L}_\alpha\cap \mathfrak{u}_\alpha^\ast$ and $\mathfrak{u}_\alpha^\ast=\mathbb{C}\{-S_\alpha^\ast\, A_\alpha^\ast\}$. Therefore
$$\textnormal{dim}_\mathbb{R}\Ann \Delta_\alpha=\left\lbrace \begin{array}{lcc} 
2 & \textnormal{if} & \mathbb{J}_\alpha=\pm \mathcal{J}_0\\
0 & \textnormal{if} & \mathbb{J}_\alpha= \mathcal{J}_\alpha.
\end{array}
\right.$$ 
This suggests  that $\mathbb{J}_\alpha=\pm \mathcal{J}_0$ has complex type, which is something that we already knew,  but 
more importantly it suggests that $\mathbb{J}_\alpha= \mathcal{J}_\alpha$ has symplectic type, that is, every invariant generalized almost complex structure of noncomplex type on $\mathfrak{u}_\alpha$ might be obtained by applying a $B$-transformation to a symplectic form on $\mathfrak{u}_\alpha$.

Let $B\in \bigwedge^2 \mathfrak{u}_\alpha^\ast$. Since our generalized almost complex structures are invariant, we may consider 
$B=-bS_\alpha^\ast\wedge A_\alpha^\ast$ where $b\in\mathbb{R}$, thus $B=\left( \begin{array}{cc}
0 & b\\
-b & 0
\end{array}%
\right)$ a skew-symmetric matrix. 
\begin{lemma} \label{Bsymp} A $B$-transformation of a generalized complex structure 
of symplectic type on $\mathfrak{u}_\alpha$ may produce a generalized complex structure of noncomplex type. 
\end{lemma}
\begin{proof}
	
Suppose that $\mathbb{J}_\alpha$ is of noncomplex type with $a_\alpha=0$, that is, $\mathbb{J}_\alpha$ has the form given in Remark \ref{SymplecticInvariantType}. Thus, we  get a new invariant generalized almost complex structure on 
$\mathfrak{u}_\alpha$ of type $k=0$ as follows.
$$e^{-B}\mathbb{J}_\alpha e^B=\left( 
\begin{array}{cc}
-\omega_\alpha^{-1}B & -\omega_\alpha^{-1}\\
\omega_\alpha+B\omega_\alpha^{-1}B& B\omega_\alpha^{-1}
\end{array}%
\right)={\tiny\left( 
\begin{array}{cccc}
bx_\alpha & 0 & 0 & -x_\alpha\\ 
0 &  bx_\alpha& x_\alpha & 0\\
0 & -b^2x_\alpha-1/x_\alpha & -bx_\alpha & 0\\
b^2x_\alpha+1/x_\alpha & 0 & 0 & -bx_\alpha
\end{array}%
\right)}.$$
Therefore, setting $a_\alpha=bx_\alpha$ and $y_\alpha=b^2x_\alpha+1/x_\alpha$ we obtain
$$e^{-B}\mathbb{J}_\alpha e^B={\tiny\left( 
\begin{array}{cccc}
a_\alpha & 0 & 0 & -x_\alpha\\ 
0 &  a_\alpha& x_\alpha & 0\\
0 & -y_\alpha & -a_\alpha & 0\\
y_\alpha & 0 & 0 & -a_\alpha
\end{array}%
\right)},$$
with $a_\alpha,x_\alpha,y_\alpha\in\mathbb{R}$ such that $a_\alpha^2=x_\alpha y_\alpha-1$.
\end{proof}

Motivated by Lemma \ref{Bsymp} we set up the following notation.

\begin{notation}\label{SymplecticAsso}
Let $\mathbb{J}_\alpha=\mathcal{J}_\alpha$ be a generalized complex structure of noncomplex type on $\mathfrak{u}_\alpha$ as in Remark \ref{canonical}. We will denote by $$\mathbb{J}_{\omega_\alpha}:=\left( 
\begin{array}{cccc}
0 & 0 & 0 & -x_\alpha\\ 
0 & 0 & x_\alpha & 0\\
0 & -1/x_\alpha & 0 & 0\\
1/x_\alpha & 0 & 0 & 0
\end{array}%
\right),$$
and call it the generalized complex structure of symplectic type induced by $\mathbb{J}_\alpha$.
\end{notation}

As consequence of Lemma \ref{Bsymp} all generalized complex structures on $\mathfrak{u}_\alpha\oplus \mathfrak{u}_\alpha^\ast$ of noncomplex type are actually $B$-symplectic generalized almost complex structures.

\begin{proposition}\label{BFieldSymplectic}
Let $\alpha$ be a positive root. Given a generalized complex structure $\mathbb{J}_\alpha$ on $\mathfrak{u}_\alpha$ of noncomplex type, there exists $B$-transformation  $B_\alpha\in \bigwedge^2 \mathfrak{u}_\alpha^\ast$ such that 
$$e^{-B_\alpha}\mathbb{J}_{\omega_\alpha} e^{B_\alpha} =\mathbb{J}_\alpha.$$
\end{proposition}
\begin{proof}
Suppose  $\mathbb{J}_\alpha={\tiny \left( 
\begin{array}{cccc}
a_\alpha & 0 & 0 & -x_\alpha\\ 
0 &  a_\alpha& x_\alpha & 0\\
0 & -y_\alpha & -a_\alpha & 0\\
y_\alpha & 0 & 0 & -a_\alpha
\end{array}%
\right)}$ with $a_\alpha^2=x_\alpha y_\alpha-1$ and define $B_\alpha=-\dfrac{a_\alpha}{x_\alpha}S_\alpha^\ast\wedge A_\alpha^\ast$. Then as in Lemma \ref{Bsymp}, the identity $a_\alpha^2=x_\alpha y_\alpha-1$ implies that $e^{-B_\alpha}\mathbb{J}_{\omega_\alpha} e^{B_\alpha} =\mathbb{J}_\alpha$. 
\end{proof}
\begin{corollary}\label{BField2}
Let $\alpha$ be a positive root. Let $\mathbb{J}_\alpha$ and $\mathbb{J}_\alpha'$ be two generalized complex structures on $\mathfrak{u}_\alpha$ of noncomplex type with
$$\mathcal{J}_{\alpha}={\tiny\left( 
\begin{array}{cccc}
a_\alpha & 0 & 0 & -x_\alpha\\ 
0 &  a_\alpha& x_\alpha & 0\\
0 & -y_\alpha & -a_\alpha & 0\\
y_\alpha & 0 & 0 & -a_\alpha
\end{array}%
\right)} \quad \textnormal{and}\quad \mathcal{J}_{\alpha}'={\tiny\left( 
\begin{array}{cccc}
b_\alpha & 0 & 0 & -x_\alpha\\ 
0 &  b_\alpha& x_\alpha & 0\\
0 & -z_\alpha & -b_\alpha & 0\\
z_\alpha & 0 & 0 & -b_\alpha
\end{array}%
\right)}.$$
Then, there exists a $B$-transformation $B_\alpha\in \bigwedge^2 \mathfrak{u}_\alpha^\ast$ such that
$$e^{-B_\alpha}\mathbb{J}_\alpha e^{B_\alpha} =\mathbb{J}_\alpha'.$$
\end{corollary}
\begin{proof}
It is clear that $\mathbb{J}_{\omega_\alpha}=\mathbb{J}_{\omega_\alpha}'$. By Proposition \ref{BFieldSymplectic}, there exist $\hat{B}_\alpha$ and $B_\alpha'$ in $\bigwedge^2 \mathfrak{u}_\alpha^\ast$ such that
$$e^{\hat{B}_\alpha}\mathbb{J}_\alpha e^{-\hat{B}_\alpha}=e^{B_\alpha'}\mathbb{J}_\alpha' e^{-B_\alpha'}.$$
As in this case the identity $\exp(\hat{B}+B')=\exp(\hat{B})\cdot\exp(B')$ holds, we get the result for $B_\alpha=B_\alpha'-\hat{B}_\alpha$.
\end{proof}
\begin{remark}
By Proposition \ref{BFieldSymplectic} we have that if $\mathbb{J}_\alpha$ is of noncomplex type for some $\alpha\in\Pi^+$, the subspaces $U^k_\alpha$ decomposing the space of 
forms $\bigwedge^\bullet (\mathfrak{u}_\alpha^\ast\otimes\mathbb{C})$  are given by (see \cite{Ca})
$$U^k_\alpha=e^{B_\alpha+i\omega_\alpha}e^{\frac{-\omega^{-1}_\alpha}{2i}} \bigwedge\!^{n-k}(\mathfrak{u}_\alpha^\ast\otimes\mathbb{C}),$$
where	$\omega_\alpha=-\dfrac{k_\alpha}{x_\alpha}\omega_{b_0}|_{\mathfrak{u}_\alpha}$ and $B_\alpha=-\dfrac{a_\alpha}{x_\alpha}S_\alpha^\ast\wedge A_\alpha^\ast$. 
\end{remark}
Let us now suppose that $\mathbb{J}_\alpha$ is of complex type (i.e. type  $k=1$), namely $\mathbb{J}_\alpha=\pm \mathcal{J}_0$. If $\mathbb{J}_\alpha=\mathcal{J}_0$
we have that
$$\mathbb{J}_\alpha=\left( 
\begin{array}{cc}
-J_0 & 0\\
0 & J_0^\ast
\end{array}%
\right)\qquad\textnormal{where}\qquad J_0=\left( 
\begin{array}{cc}
0 & 1 \\
-1  & 0
\end{array}%
\right).$$
If we apply a $B$-transform to $\mathbb{J}_\alpha$, we get another generalized complex structure on $\mathfrak{u}_\alpha$ of type $k=1$. It is easy to check that for every $B\in \bigwedge^2 \mathfrak{u}_\alpha^\ast$
$$e^{-B}\mathbb{J}_\alpha e^B=\left( 
\begin{array}{cc}
-J_0 & 0\\
0 & J^\ast_0
\end{array}%
\right)=\mathbb{J}_\alpha=\mathcal{J}_0,$$
because $BJ_0+J^\ast_0 B=0$.
\begin{proposition}\label{BFieldComplex}
	Let $\alpha$ be a positive root. If $\mathbb{J}_\alpha$ is a generalized complex structure of complex type on $\mathfrak{u}_\alpha$, then for every $B$-transformation $B\in \bigwedge^2 \mathfrak{u}_\alpha^\ast$
	$$e^{-B}\mathbb{J}_{\alpha} e^{B} =\mathbb{J}_\alpha.$$
\end{proposition}
In other words, generalized complex structures of complex type on $\mathfrak{u}_\alpha$ are fixed points by the $B$-transformation action.

\begin{remark}
By Proposition \ref{BFieldComplex}, the subspaces $U^k_\alpha$ decomposing the space of 
forms $\bigwedge^\bullet (\mathfrak{u}_\alpha^\ast\otimes\mathbb{C})$  are given by (see \cite{G1})
$$U^k_\alpha=\bigoplus_{p-q=k}\bigwedge^{p,q}\mathfrak{u}_\alpha^\ast,$$
where $\bigwedge^{p,q}\mathfrak{u}_\alpha^\ast$ is the standard $(p,q)$-decomposition of forms of a complex space.
\end{remark}

We can now describe the effect of  $B$-transformations on $\mathcal M_\alpha$.
\begin{proposition}\label{AmoduloB} Let $\alpha$ be a positive root. The set  $\displaystyle \mathfrak{M}_\alpha(\mathbb{F})=\frac{\mathcal M_\alpha}{B}$ of equivalence classes of generalized complex structures on $\mathfrak{u}_\alpha$ modulo 
the action of $B$-transformations consists of 2 disjoint sets:
\begin{enumerate}

\item[$\iota$.] a punctured real line $\mathbb{R}^\ast$ parametrizing structures of symplectic type, and 

\item[$\iota\iota$.] 2 extra points $\pm \mathbf{0}$ corresponding to $\pm \mathcal{J}_0$.

\end{enumerate}
\end{proposition}
\begin{proof}
The structures $\pm \mathcal{J}_0$ are fixed points of the action by $B$-transformations. So, we need only to consider structures of noncomplex type. By Lemma \ref{Aalpha}, structures in $\mathcal M_\alpha$ which are 
of noncomplex type are parameterized by a real surface $a_\alpha^2 -x_\alpha y_\alpha=-1$ in $\mathbb{R}^3$, and by Proposition \ref{BFieldSymplectic} every point on this surface
is the image of a generalized complex structure of symplectic type by a $B$-transformation determined by $x_\alpha$. 
Hence the quotient of this real surface by $B$-transformations reduces to a punctured line
holding the values of $x_\alpha$. Since by Remark \ref{SymplecticInvariantType} we have that generalized complex structures of symplectic type corresponding to parameters  $x_\alpha \neq x_\alpha'$ are distinct.
So, points on this  real line represent  inequivalent classes. 
\end{proof}

\begin{remark}\label{topology1} Note that  the parameters $x_\alpha   \in \mathbb{R}^\ast$  represent structures of noncomplex type $\mathcal J_\alpha$,
 as in Remark \ref{canonical},  and
such structures are matrices which have entries
 12, 21, 34, and 43 all zero. Therefore a sequence of structures of noncomplex type can not converge to either one of the matrices $\pm \mathcal J_0$.
As a consequence, the points $\pm \mathbf{0}$ are isolated in $ \mathfrak{M}_\alpha(\mathbb{F})$.
\end{remark}

\subsection{Moduli space}
Now we want to use the results obtained in subsection \ref{SB:5.1} for clarifying which are the effects of the action by invariant $B$-transformations globally, that is, when we consider the action on the set of all invariant generalized almost complex structures globally on $\mathbb{F}$. For simplicity suppose that $\Pi^+=\{\alpha_1,\alpha_2\}$ and let $\mathbb{J}$ and $\mathbb{J}'$ be two invariant generalized almost complex structure on $\mathbb{F}$. Let us also assume that for $j=1,2$ we have that $\mathbb{J}_{\alpha_j}$ and $\mathbb{J}_{\alpha_j}'$ are of noncomplex type with
$$\mathcal{J}_{\alpha_j}=\left( 
\begin{array}{cc}
\mathcal{A}_{\alpha_j} & \mathcal{X}_{\alpha_j} \\
\mathcal{Y}_{\alpha_j}  & -\mathcal{A}_{\alpha_j}
\end{array}%
\right)\quad \textnormal{and}\quad  \mathcal{J}_{\alpha_j}'=\left( 
\begin{array}{cc}
\mathcal{A}_{\alpha_j}' & \mathcal{X}_{\alpha_j} \\
\mathcal{Y}_{\alpha_j}'  & -\mathcal{A}_{\alpha_j}'
\end{array}%
\right).$$
By Corollary \ref{BField2}, there exist $B$-transformations $B_{\alpha_1}\in\bigwedge^2\mathfrak{u}_{\alpha_1}^\ast$ and $B_{\alpha_2}\in\bigwedge^2\mathfrak{u}_{\alpha_2}^\ast$   such that
$$e^{-B_{\alpha_1}}\mathbb{J}_{\alpha_1} e^{B_{\alpha_1}} =\mathbb{J}_{\alpha_1}'\qquad\textnormal{and}\qquad e^{-B_{\alpha_2}}\mathbb{J}_{\alpha_2} e^{B_{\alpha_1}} =\mathbb{J}_{\alpha_2}'.$$
These identities imply that
$$\mathcal{J}_{\alpha_j}'=\left( 
\begin{array}{cc}
\mathcal{A}_{\alpha_j}+\mathcal{X}_{\alpha_j}B_j & \mathcal{X}_{\alpha_j} \\
-\mathcal{A}_{\alpha_j}B_j -B_j\mathcal{A}_{\alpha_j}+\mathcal{Y}_{\alpha_j}-B_j\mathcal{X}_{\alpha_j}B_j & -B_j\mathcal{X}_{\alpha_j}-\mathcal{A}_{\alpha_j}
\end{array}%
\right)=\left( 
\begin{array}{cc}
\mathcal{A}_{\alpha_j}' & \mathcal{X}_{\alpha_j} \\
\mathcal{Y}_{\alpha_j}'  & -\mathcal{A}_{\alpha_j}'
\end{array}%
\right).$$
Define $B\in\bigwedge^2\mathfrak{n}^-$ as $B=\left( 
\begin{array}{cc}
	B_1 & 0 \\
	0  & B_2
\end{array}%
\right)$. Setting 
$$\mathbb{J}=\left( 
\begin{array}{cccc}
\mathcal{A}_{\alpha_1} & 0 & \mathcal{X}_{\alpha_1} & 0\\ 
0 &  \mathcal{A}_{\alpha_2} & 0 & \mathcal{X}_{\alpha_2}\\
\mathcal{Y}_{\alpha_1} & 0 & -\mathcal{A}_{\alpha_1} & 0\\
0 & \mathcal{Y}_{\alpha_2} & 0 & -\mathcal{A}_{\alpha_2}
\end{array}%
\right),$$
we get that
$$e^{-B}\mathbb{J}e^{B}={\tiny\left( 
\begin{array}{cccc}
1 & 0 & 0 & 0\\ 
0 & 1 & 0 & 0\\
-B_1 & 0 & 1 & 0\\
0 & -B_2 & 0 & 1
\end{array}%
\right)\left( 
\begin{array}{cccc}
\mathcal{A}_{\alpha_1} & 0 & \mathcal{X}_{\alpha_1} & 0\\ 
0 &  \mathcal{A}_{\alpha_2} & 0 & \mathcal{X}_{\alpha_2}\\
\mathcal{Y}_{\alpha_1} & 0 & -\mathcal{A}_{\alpha_1} & 0\\
0 & \mathcal{Y}_{\alpha_2} & 0 & -\mathcal{A}_{\alpha_2}
\end{array}%
\right)\left( 
\begin{array}{cccc}
1 & 0 & 0 & 0\\ 
0 & 1 & 0 & 0\\
B_1 & 0 & 1 & 0\\
0 & B_2 & 0 & 1
\end{array}%
\right)}=\mathbb{J}'.$$
Let us now suppose that $\mathbb{J}_{\alpha_1}=\mathbb{J}_{\alpha_1}'=\mathcal{J}_0$ (or $-\mathcal{J}_0$) are of complex type and $\mathbb{J}_{\alpha_2}$  and $\mathbb{J}_{\alpha_2}'$ are of noncomplex type with
$$\mathcal{J}_{\alpha_2}=\left( 
\begin{array}{cc}
\mathcal{A}_{\alpha_1} & \mathcal{X}_{\alpha_2} \\
\mathcal{Y}_{\alpha_2}  & -\mathcal{A}_{\alpha_2}
\end{array}%
\right),\qquad  \mathcal{J}_{\alpha_2}'=\left( 
\begin{array}{cc}
\mathcal{A}_{\alpha_2}' & \mathcal{X}_{\alpha_2} \\
\mathcal{Y}_{\alpha_2}'  & -\mathcal{A}_{\alpha_2}'
\end{array}%
\right).$$
Denote 
$\mathbb{J}_{\alpha_1}=\mathbb{J}_{\alpha_1}'=\left( 
\begin{array}{cc}
J_0 & 0 \\
0  & J_0
\end{array}%
\right).$
By Corollary \ref{BField2} and Proposition \ref{BFieldComplex} we get that there exist $B$-transformations $B_1\in\bigwedge^2\mathfrak{u}_{\alpha_1}^\ast$ and $B_2\in\bigwedge^2\mathfrak{u}_{\alpha_2}^\ast$   such that
$$e^{-B_1}\mathbb{J}_{\alpha_1} e^{B_1} =\mathbb{J}_{\alpha_1}'\qquad e^{-B_2}\mathbb{J}_{\alpha_2} e^{B_2} =\mathbb{J}_{\alpha_2}'.$$
Setting
$$\mathbb{J}=\left( 
\begin{array}{cccc}
J_0 & 0 & 0 & 0\\ 
0 &  \mathcal{A}_{\alpha_2} & 0 & \mathcal{X}_{\alpha_2}\\
0 & 0 & J_0 & 0\\
0 & \mathcal{Y}_{\alpha_2} & 0 & -\mathcal{A}_{\alpha_2}
\end{array}%
\right),$$
it is easy to check that for $B=\left( 
\begin{array}{cc}
B_1 & 0 \\
0  & B_2
\end{array}%
\right)$ we obtain
$$e^{-B}\mathbb{J}e^{B}={\tiny\left( 
\begin{array}{cccc}
1 & 0 & 0 & 0\\ 
0 & 1 & 0 & 0\\
-B_1 & 0 & 1 & 0\\
0 & -B_2 & 0 & 1
\end{array}%
\right)\left( 
\begin{array}{cccc}
J_0 & 0 & 0 & 0\\ 
0 &  \mathcal{A}_{\alpha_2} & 0 & \mathcal{X}_{\alpha_2}\\
0 & 0 & J_0 & 0\\
0 & \mathcal{Y}_{\alpha_2} & 0 & -\mathcal{A}_{\alpha_2}
\end{array}%
\right)\left( 
\begin{array}{cccc}
1 & 0 & 0 & 0\\ 
0 & 1 & 0 & 0\\
B_1 & 0 & 1 & 0\\
0 & B_2 & 0 & 1
\end{array}%
\right)}=\mathbb{J}'.$$
When $\mathbb{J}_{\alpha_1}=\mathbb{J}_{\alpha_1}'=-\mathcal{J}_0$ the reasoning is analogous.

Summing up, by an inductive process we have
\begin{theorem}\label{GlobalBField}
Let $\mathbb{J}$ and $\mathbb{J}'$ be two invariant generalized almost complex structures on $\mathbb{F}$ such that for each $\alpha\in \Pi^+$ the following conditions hold
\begin{enumerate}
\item[$\iota$.] if $\mathbb{J}_\alpha$ is of complex type, then $\mathbb{J}_\alpha=\mathbb{J}_\alpha'$, and
\item[$\iota\iota$.] if $\mathbb{J}_\alpha$ is of noncomplex type, then $\mathbb{J}_\alpha'$ is also of noncomplex type with
$$\mathcal{J}_{\alpha}=\left( 
\begin{array}{cc}
\mathcal{A}_{\alpha} & \mathcal{X}_{\alpha} \\
\mathcal{Y}_{\alpha}  & -\mathcal{A}_{\alpha}
\end{array}%
\right) \quad \textnormal{and}\quad \mathcal{J}_{\alpha}'=\left( 
\begin{array}{cc}
\mathcal{A}_{\alpha}' & \mathcal{X}_{\alpha} \\
\mathcal{Y}_{\alpha}'  & -\mathcal{A}_{\alpha}'
\end{array}%
\right).$$
\end{enumerate}
Then there exists an invariant $B$-transformation $B\in \wedge^2 (\mathfrak{n}^-)^\ast$ such that
$$e^{-B}\mathbb{J}e^{B}=\mathbb{J}'.$$
\end{theorem}
\begin{proof}
Suppose that $\Pi^+=\{\alpha_1,\cdots,\alpha_d \}$. For each $\alpha_j\in \Pi^+$ with $j=1,2,\cdots,d$ we set 
$$\mathbb{J}_{\alpha_j}=\left( 
\begin{array}{cc}
\mathcal{A}_{\alpha_j} & \mathcal{X}_{\alpha_j} \\
\mathcal{Y}_{\alpha_j}  & -\mathcal{A}_{\alpha_j}
\end{array}%
\right)  \quad \textnormal{and}\quad  \mathbb{J}_{\alpha_j}'=\left( 
\begin{array}{cc}
\mathcal{A}_{\alpha_j}' & \mathcal{X}_{\alpha_j} \\
\mathcal{Y}_{\alpha_j}'  & -\mathcal{A}_{\alpha_j}'
\end{array}%
\right).$$
where by hypotheses if $\mathbb{J}_{\alpha_j}$ is the complex type, then $\mathbb{J}_{\alpha_j}=\mathbb{J}_{\alpha_j}'=\pm \mathcal{J}_0$ which implies that $\mathcal{A}_{\alpha_j}=\mathcal{A}_{\alpha_j}'=-\mathcal{A}_{\alpha_j}=-\mathcal{A}_{\alpha_j}'=\pm J_0$ and $\mathcal{X}_{\alpha_j}=\mathcal{Y}_{\alpha_j}'=\mathcal{Y}_{\alpha_j}=0$. Otherwise, if $\mathbb{J}_{\alpha_j}$ is of noncomplex type, then $\mathbb{J}_{\alpha_j}'$ is also of noncomplex type.

By Corollary \ref{BField2} and Proposition \ref{BFieldComplex}, for all $j=1,2,\cdots,d$ there exist $B$-transformations $B_j\in\bigwedge^2\mathfrak{u}_{\alpha_j}^\ast$ such that
\begin{equation}\label{tired}
e^{-B_j}\mathbb{J}_{\alpha_j} e^{B_j} =\mathbb{J}_{\alpha_j}'.
\end{equation}
If we define $B\in\bigwedge^2(\mathfrak{n}^-)^\ast$ by
$B={\tiny \left( 
\begin{array}{cccc}
B_{1} &  &  & \\ 
&  B_{2}&  & \\
&  & \ddots & \\
&  &  & B_{d}
\end{array}%
\right)}$, then setting
$$\mathbb{J}={\tiny\left( 
\begin{array}{ccccccccc}
\mathcal{A}_{\alpha_1} &  &  &  & \mathcal{X}_{\alpha_1} & & \\
 & \mathcal{A}_{\alpha_2} &  &  & &  \mathcal{X}_{\alpha_2}& \\
 
  &  & \ddots & & &   &\ddots \\
   &  &  & \mathcal{A}_{\alpha_d} & &  & & \mathcal{X}_{\alpha_d} \\
  \mathcal{Y}_{\alpha_1}   & &  &  & -\mathcal{A}_{\alpha_1} & &\\
    & \mathcal{Y}_{\alpha_2} &  &  &  & -\mathcal{A}_{\alpha_2} &\\
     &  & \ddots & & &   &\ddots\\
     &  &  &\mathcal{Y}_{\alpha_d} & &   & & -\mathcal{A}_{\alpha_d}\\
     
\end{array}%
\right)},$$
we get that the identities deduced from \eqref{tired} imply that $e^{-B}\mathbb{J}e^{B}=\mathbb{J}'$.
\end{proof}

\begin{remark}
In Theorem \eqref{GlobalBField} we have that $B$ is initially defined on $T_{b_0}\mathbb{F}=\mathfrak{n}^-$. Thus we can use the adjoint action on $\mathbb{F}$ to define an invariant 2-form $\widetilde{B}$ on $\mathbb{F}$ as follows. If $x=gT\in \mathbb{F}=U/T$ for some $g\in U$ we define
$$\widetilde{B}_x(X(x),Y(x))=B(\Ad(g^{-1})_{\ast,x}X(x),\Ad(g^{-1})_{\ast,x}Y(x)),\qquad X(x), Y(x)\in T_x\mathbb{F}.$$ 
It is simple to check using the rule chain that $\widetilde{B}\in\Omega^2(\mathbb{F})$ is invariant by the adjoint action in the sense that $(\Ad(g))^\ast \widetilde{B}=\widetilde{B}$ for all $g\in U$. Analogously, we can define $\mathbb{J}$ and $\mathbb{J}'$ on $\mathbb{TF}$ such that they are invariant by $(\Ad\oplus \Ad^\ast)(g)$ for all $g\in T$ and satisfy $e^{-\widetilde{B}}\mathbb{J}e^{\widetilde{B}}=\mathbb{J}'$.
\end{remark}

We denote by $\mathfrak{M}_a(\mathbb{F})$ the quotient space obtained from $\mathcal{M}_a$ module 
the action by $B$-transformations (see Definition \ref{Bmoduli}). We also identify $\pm \mathcal{J}_0$ with the two points $\pm \mathbf{0}$.
\begin{corollary}\label{Bmoduli1}
Suppose that $\vert \Pi^+ \vert=d$. Then
$$\mathfrak M_a (\mathbb F)= \prod_{\alpha\in\Pi^+} \mathfrak{M}_\alpha(\mathbb{F})=(\mathbb R^\ast \cup \pm \mathbf{0})_{\alpha_1} \times \cdots \times (\mathbb  R^\ast\cup\pm \mathbf{0})_{\alpha_d},$$
where $\alpha_j\in\Pi^+$.  Moreover, in the product topology, $\mathfrak{M}_a(\mathbb{F})$ contains an open dense subset $ (\mathbb{R}^\ast)^d$ parametrizing all invariant generalized almost complex structures of symplectic type and there are exactly $2^d$  isolated points corresponding to the structures of complex type.
\end{corollary}
\begin{remark}
	If $\mathbb{J}$ is an invariant generalized almost complex structure on a maximal flag $\mathbb{F}$ then it is regular in the sense that its type at every point of $\mathbb{F}$ is the same that the type at the origin $b_0$ of $\mathbb{F}$ because of  invariance.
\end{remark}

\subsection{$\Theta$-Stratification} 
Let $\Sigma$ be a simple root system for $\mathfrak g$. 
We now  explain how subsets $\Theta \subseteq \Sigma$ determine 
families of invariant generalized complex structures in $\mathfrak M (\mathbb F)$.

\begin{notation}\label{VecNotation}
	 As shown in Theorem \ref{GlobalBField}, every class $\sigma$ in $\mathfrak M_a (\mathbb F)$ 
	 may be represented by a vector of the form $\sigma=(\sigma_1,\sigma_2,\cdots,\sigma_d)$. If $\sigma$ is 
	  of type $k$ then there exists a $k$-multi-index $I \subset \lbrace 1,2,\cdots,d\rbrace$
	   such that  $\sigma_i \in \lbrace \mathbf{0},-\mathbf{0} \rbrace$ for all $i\in I$ and 
	  $\sigma_i \in \mathbb{R}^\ast$ for  $i \notin  I$ ($d-k$ elements satisfies this).
\end{notation}

Consider   $\Theta \subseteq \Sigma$ containing $r=\vert \Theta \vert$ elements.
Let  $\mathbb{J}$ be an invariant generalized complex structure on $\mathbb{F}$ that  is of noncomplex type for each $\alpha \in \langle \Theta \rangle^+$ and of complex type for each $\alpha \in \Pi^+ \backslash \langle \Theta \rangle^+$ as in  Theorem \ref{theta2}.
 Then $\mathbb{J}$ is of type  $k=d-\vert\langle\Theta \rangle^+\vert$.
 Another invariant generalized complex structure on $\mathbb{F}$ which is equivalent to $\mathbb{J}$, that is, which belongs
  to the same $B$-field transformation orbit,  has the same type as $\mathbb{J}$, see Remark \ref{int}.
   Therefore, the type $k$ is well defined for classes in $\mathfrak M_a (\mathbb F)$ and for classes in $\mathfrak M (\mathbb F)$.

Let $\sigma$ be the vector  that represents $\mathbb{J}$. We wish to consider the subset of $\mathfrak M (\mathbb F)$ formed 
by all classes of type $k$ which have the same $\sigma_i$ entries for all $i\in I$.
 Since every element in $\langle \Theta \rangle^+$ is the sum of elements in $\Theta$, second equation of \eqref{Integrabilityconditions} implies that the number of independent entries  $\lbrace \sigma_i :\ i\notin I\rbrace$ is $r$, and these are free to vary inside $(\mathbb R^\ast)^r$, thus forming 
 $r$-dimensional families of invariant generalized complex on $\mathbb{F}$ of type $k$. 

\begin{lemma}\label{ThetaCells}
Every subset $\Theta \subseteq \Sigma$ with $r=\vert \Theta \vert$ determines an $r$-dimensional family inside $\mathfrak M (\mathbb F)$ of invariant generalized complex structures on $\mathbb{F}$ of type $k=d-\vert\langle\Theta \rangle^+\vert$.
\end{lemma}

We will call to these $r$-dimensional families $\Theta$-cells. In particular, if $\Theta=\emptyset$ we get points which represent invariant generalized complex structures of complex type and if $\Theta=\Sigma$ with $r=\vert \Sigma \vert$ we get an $r$-dimensional family of invariant generalized complex structures of symplectic type.

\begin{notation}\label{CellsNotation}
Fix a choice of positive roots $\Pi^+$ with corresponding simple root system $\Sigma$ 
and consider all values of $i$ with  $1\leq i\leq d$. As consequence of Theorems \ref{theta} and \ref{theta2},
there are   {\it standard} cells of dimension $i$ obtained by 
a choice $\Theta \subseteq \Sigma$ of simple positive roots, which we denote  by
 ${\bf c}^i$, so that, all  $\Theta$-cells of dimension $i$ will be obtained from the standard ones via the action of the Weyl group. 
If there are no 
$\Theta$-cells of a given dimension $i$, the corresponding choice of ${\bf c}^i$ is $\emptyset$. For each element $w$ in the Weyl group $\mathcal{W}$, we denote by $w\cdot {\bf c}^i$ the $\Theta$-cell corresponding to $w \cdot \Theta \subset w \cdot \Sigma$.
\end{notation}

It is well known that the Weyl group $\mathcal{W}$ acts transitively on Weyl chambers and hence on the set of simple roots systems of $\mathfrak{g}$. Thus, as consequence of Lemma \ref{ThetaCells} and with Notation \ref{CellsNotation} we may give a decomposition of $\mathfrak M (\mathbb F)$ in $\Theta$-cells as follows:
\begin{theorem}\label{IntModuli}
The moduli space of invariant generalized complex structures on $\mathbb F$ admits the following  cell decomposition:
	$$\mathfrak M (\mathbb F)=\bigsqcup_{i=1}^{d} \bigsqcup_{w\in \mathcal W} w\cdot {\bf c}^i.$$
\end{theorem}

\subsection{Invariant pure spinor lines} 
We now  exhibit the
pure spinor lines associated to invariant generalized almost complex structures $\mathbb{J}$ on
$\mathbb{F}$. We will use the following facts:
\begin{remark}\cite{H},\cite{G1}
Suppose that $(M,\mathbb{J})$ and $(M',\mathbb{J}')$ are two generalized almost complex manifolds. Then $\mathbb{J}\oplus \mathbb{J}'$ is a generalized almost complex structure on $M\times M'$ which satisfies
\begin{enumerate}
\item[$\iota$.] if $\mathcal{L}\subset \mathbb{T}M\otimes \mathbb C$ and $\mathcal{L}'\subset \mathbb{T}M'\otimes \mathbb C$ are the maximal isotropic subbundles determined by $\mathbb{J}$ and $\mathbb{J}'$, respectively, then the corresponding maximal isotropic subbundle in $\mathbb{T}(M\times M')\otimes \mathbb{C}$ determined by $\mathbb{J}\oplus \mathbb{J}'$ is $\pi_1^\ast(\mathcal{L})\oplus \pi_2^\ast(\mathcal{L}')$.
\item[$\iota\iota$.] If $\varphi_\mathcal{L}$ and $\varphi'_\mathcal{L'}$ are pure spinors determined by $\mathbb{J}$ and $\mathbb{J}'$, respectively, then the pure spinor line determined by $\mathbb{J}\oplus \mathbb{J}'$ is generated by
$\pi_1^\ast(\varphi_\mathcal{L})\wedge \pi_2^\ast(\varphi_\mathcal{L}')$.
\item[$\iota\iota\iota$.] The type of $\mathbb{J}\oplus \mathbb{J}'$ is additive, that is, $\textnormal{type}(\mathbb{J}\oplus \mathbb{J}')_{(p,q)}=\textnormal{type}(\mathbb{J})_p+\textnormal{type}(\mathbb{J}')_q$ for all $(p,q)\in M\times M'$.
\end{enumerate}
It is simple to generalize this result to  the product of a finite number of generalized almost complex manifolds, using induction.
\end{remark}

An invariant generalized almost complex structure $\mathbb{J}:\mathfrak{n}^-\oplus(\mathfrak{n}^-)^\ast\to \mathfrak{n}^-\oplus(\mathfrak{n}^-)^\ast$ on $\mathbb{F}$ may be written as $$\mathbb{J}=\bigoplus_{\alpha\in\Pi^+}\mathbb{J}_\alpha,$$
where $\mathfrak{n}^-=\bigoplus_{\alpha\in\Pi^+}\mathfrak{u}_\alpha$ and $\mathbb{J}_\alpha:\mathfrak{u}_\alpha\oplus\mathfrak{u}^\ast_\alpha\to \mathfrak{u}_\alpha\oplus\mathfrak{u}^\ast_\alpha$ is described in Remark \ref{canonical}. The 
corresponding invariant maximal isotropic subspace of $(\mathfrak{n}^-\oplus(\mathfrak{n}^-)^\ast)\otimes \mathbb{C}$ is 
 $\mathcal{L}=\bigoplus_{\alpha\in\Pi^+}\mathcal{L}_\alpha$ where $\mathcal{L}_\alpha$ is determined by \eqref{isotropic1}.
Using Propositions \ref{BFieldSymplectic} and \ref{BFieldComplex} we have:

\begin{lemma}\label{SpinorRoots}
Let $\alpha$ be a positive root. 
\begin{enumerate}
\item[$\iota.$] If $\mathbb{J}_\alpha$ is of complex type with $\mathbb{J}_\alpha=\mathcal{J}_0$ then its pure spinor line  is generated by a 
$(1,0)$ form $\Omega_\alpha \in  \wedge^{1,0}\mathfrak{u}_\alpha^\ast$ which defines a complex structure on $\mathfrak{u}_\alpha$. Otherwise, when $\mathbb{J}_\alpha=-\mathcal{J}_0$  we have a pure spinor line generated by $\overline{\Omega_\alpha} \in  \wedge^{0,1}\mathfrak{u}_\alpha^\ast$.
\item[$\iota\iota$.] If $\mathbb{J}_\alpha$ is of noncomplex type, its pure spinor line  is generated by $$e^{B_\alpha+i\omega_\alpha},$$
where $B_\alpha=-\dfrac{a_\alpha}{x_\alpha}S_\alpha^\ast\wedge A_\alpha^\ast$ and $\omega_\alpha=-\dfrac{k_\alpha}{x_\alpha}\omega_{b_0}|_{\mathfrak{u}_\alpha}$ with $\omega_{b_0}$ denoting the KKS symplectic form on $\mathfrak{n}^-$ and $k_\alpha=\dfrac{1}{2i\langle H,H_\alpha\rangle}$.
\end{enumerate}
\end{lemma}
Let us decompose the set of positive roots as $\Pi^+= \Theta_{nc}\sqcup\Theta_{c}$ where $\mathbb{J}_\alpha=\mathcal{J}_\alpha$ is of noncomplex type for all $\alpha\in \Theta_{nc}$ and $\mathbb{J}_\alpha=\pm \mathcal{J}_0$ is of complex type for all $\alpha\in \Theta_{c}$.
\begin{remark}
$\textnormal{type}(\mathbb{J}):=\textnormal{type}(\mathbb{J})_{b_0}=\vert \Theta_{c} \vert$.
\end{remark}
Clearly, when $\Theta_{c}=\emptyset$ we get that $\mathbb{J}$ is a $B$-transform of a structure of symplectic type and when $\Theta_{c}=\Pi^+$ we obtain that $\mathbb{J}$ is of complex type. Otherwise, when $\Theta_{c}\subset\Pi^+$ we get a generalized almost complex structure which is neither complex nor 
symplectic.

Thus, we may explicitly write the invariant pure spinor associated to each invariant generalized almost complex structure $\mathbb{J}$ as follows.

\begin{proposition}\label{Spinor}
Let $\mathbb{J}$ be an invariant generalized almost complex structure on a maximal flag $\mathbb{F}$. Then the invariant pure spinor line $K_\mathcal{L}<\bigwedge^\bullet (\mathfrak{n}^-)^\ast\otimes \mathbb{C}$ associated to $\mathbb{J}$ is generated by
	$$\varphi_\mathcal{L}=e^{\sum_{\alpha\in \Theta_{nc}}(B_\alpha +i\omega_\alpha)}\bigwedge_{\alpha\in \Theta_{c}}\Omega_\alpha,$$
	where $B_\alpha=-\dfrac{a_\alpha}{x_\alpha}S_\alpha^\ast\wedge A_\alpha^\ast$ and $\omega_\alpha=-\dfrac{k_\alpha}{x_\alpha}\omega_{b_0}|_{\mathfrak{u}_\alpha}$ with $k_\alpha=\dfrac{1}{2i\langle H,H_\alpha\rangle}$ for all $\alpha\in \Theta_{nc}$, and $\Omega_\alpha\in \wedge^{1,0}\mathfrak{u}_\alpha^\ast$ defines a complex structure on $\mathfrak{u}_\alpha$ for all $\alpha\in \Theta_{c}$.
\end{proposition}
\begin{remark}
As $\varphi_\mathcal{L}$ is initially defined on $T_{b_0}\mathbb{F}=\mathfrak{n}^-$ we can also use  invariance to define $\varphi_\mathcal{L}$ on $\mathbb{F}$ as follows. Assume that $\varphi_\mathcal{L}\in \bigwedge^k (\mathfrak{n}^-)^\ast\otimes \mathbb{C}$. Thus, at $x=gT\in \mathbb{F}=U/T$ for some $g\in U$ we define
$$\widetilde{\varphi_\mathcal{L}}_x(X_1(x),\cdots, X_k(x))=\varphi_\mathcal{L}(\Ad(g^{-1})_{\ast,x}X_1(x),\cdots,\Ad(g^{-1})_{\ast,x}X_k(x)).$$
In this case $\widetilde{\varphi_\mathcal{L}}\in \bigwedge^k T^\ast\mathbb{F}\otimes \mathbb{C}$ would be a pure spinor determined by $\mathbb{J}: \mathbb{TF}\to \mathbb{TF}$ which satisfies $(\Ad(g))^\ast\widetilde{\varphi_\mathcal{L}}=\widetilde{\varphi_\mathcal{L}}$ for all $g\in U$.
\end{remark}
Now we want to give a brief description of the effects of the action by invariant $B$-transformations on the space of invariant pure spinors on $\mathbb{F}$. The quotient space that we will get is naturally isomorphic to $\mathfrak{M}_a(\mathbb{F})$. 
\begin{remark}
Let $\textnormal{PSpin}(M)\subset\bigwedge^\bullet T^\ast M\otimes \mathbb{C}$ denote the set of pure spinors on $M$. It is simple to check that the map $\mathcal{B}\times \textnormal{PSpin}(M)\to \textnormal{PSpin}(M)$ given by $e^B \cdot \varphi=e^B\varphi$ is a well defined action of $\mathcal{B}$ on $\textnormal{PSpin}(M)$. Suppose that $\mathbb{J}$ and $\mathbb{J}'$ are the generalized almost complex structures on $M$ associated with the pure spinors
 $\varphi$ and $\varphi'$, respectively. Then, it is simple to see that $\varphi$ and $\varphi'$ belong to the same orbit by this action if and only if there exists $B\in\Omega^2(M)$ such that $e^{-B}\mathbb{J}e^B=\mathbb{J}'$. We will denote by $$\mathfrak{M}_s:=\textnormal{PSpin}(M)/\mathcal{B}$$
  the quotient space 
 of this action. Such quotient space should be naturally isomorphic to the quotient space $\mathfrak{M}_a$ of Definition \ref{Bmoduli}.
\end{remark}
Let us now consider a maximal flag manifold $\mathbb{F}$ and let $\alpha$ be a positive root. Denote by $\textnormal{PSpin}(\mathbb{F})_\alpha$ the set of invariant elements in $\textnormal{PSpin}(\mathbb{F})$ restricted to $\mathfrak{u}_\alpha$. This set is composed by those elements described in Lemma \ref{SpinorRoots}. Using  arguments similar to the ones
 in Proposition \ref{AmoduloB} we get that Corollary \ref{BField2} and Lemma \ref{SpinorRoots} imply
$$\mathfrak{M}_s(\mathbb{F})_\alpha=\textnormal{PSpin}(\mathbb{F})_\alpha/\mathcal{B}=\mathbb{R}^\ast\cup \lbrace \mathbf{s},\overline{\mathbf{s}}\rbrace.$$
		Here we are representing $\Omega_\alpha$ by $\mathbf{s}$ and $\overline{\Omega_\alpha}$ by $\overline{\mathbf{s}}$. 
		Furthermore, Theorem \ref{GlobalBField} gives: 
\begin{proposition}
Suppose that $\vert \Pi^+ \vert=d$. Then
$$\mathfrak {M}_s(\mathbb F)= \prod_{\alpha\in\Pi^+} \mathfrak{M}_s(\mathbb{F})_\alpha=(\mathbb R^\ast \cup \lbrace \mathbf{s},\overline{\mathbf{s}}\rbrace)_{\alpha_1} \times \cdots \times (\mathbb  R^\ast\cup\lbrace \mathbf{s},\overline{\mathbf{s}}\rbrace)_{\alpha_d}.$$
\end{proposition}

		Observe that in $\mathfrak{M}_s(\mathbb{F})_\alpha$ the classes of pure spinors determined by structures of noncomplex type are parametrized by $\mathbb{R}^\ast$. When we consider the limit $x_\alpha\mapsto +\infty$ we have that both $B_\alpha \mapsto 0$ and $\omega_\alpha \mapsto 0$. The same is true if $x_\alpha\mapsto -\infty$. Thus, we get that $\lim\limits_{x_\alpha\mapsto \pm\infty} e^{B_\alpha+i\omega_\alpha}=1$. 
	But, in such case we would obtain
	$\varphi_\mathcal{L}=\bigwedge_{\alpha\in \Theta_{c}}\Omega_\alpha$, which does not define a generalized complex 
	structure on $\mathbb F$. On the other hand, clearly the pure spinor $\varphi_\mathcal{L}$ which generates the pure spinor line
can not be zero. So that, here too, we obtain that each for each $\alpha$ we have that $ \mathbf{s}$ and $\overline{\mathbf{s}}$ are isolated points.
Finally, taking the product topology, we have a natural isomorphisms between 
$\mathfrak {M}_s(\mathbb F)$ and $\mathfrak {M}_a(\mathbb F)$.

\section{Invariant generalized K\"ahler structures}\label{S:6}

In this section we characterize  invariant generalized K\"ahler structures on $\mathbb{F}$.
With this in mind, first we exhibit  invariant generalized almost K\"ahler structures on $\mathbb{F}$. 
Given the requirement of invariance,  it suffices to know what happens at the origin $b_0$ of $\mathbb{F}$. Thus,
we need to find pairs of commuting invariant  generalized almost 
complex structures $(\mathbb{J},\mathbb{J'})$ such that $G:=-\mathbb{J}\mathbb{J'}$ 
defines a positive definite metric on $\mathfrak{n}^-\oplus (\mathfrak{n}^-)^\ast$. As every invariant generalized complex structure $\mathbb{J}:\mathfrak{n}^-\oplus(\mathfrak{n}^-)^\ast\to \mathfrak{n}^-\oplus(\mathfrak{n}^-)^\ast$ on $\mathbb{F}$ may be written as $$\mathbb{J}=\bigoplus_{\alpha\in\Pi^+}\mathbb{J}_\alpha,$$
where $\mathfrak{n}^-=\bigoplus_{\alpha\in\Pi^+}\mathfrak{u}_\alpha$ and $\mathbb{J}_\alpha:\mathfrak{u}_\alpha\oplus\mathfrak{u}^\ast_\alpha\to \mathfrak{u}_\alpha\oplus\mathfrak{u}^\ast_\alpha$ is described in Remark \ref{canonical}, we just need to determine those commuting pairs $(\mathbb{J}_\alpha,\mathbb{J'}_\alpha)$ such that $G_\alpha:=-\mathbb{J}_{\alpha}\mathbb{J}_{\alpha}'$ defines a positive definite metric on $\mathfrak{u}_\alpha\oplus\mathfrak{u}_\alpha^\ast$ for every positive root $\alpha$.

 We will use the following observation:
\begin{remark}\cite{G2}\label{diagonalBaction}
	Let $(\mathbb{J},\mathbb{J}')$ be a generalized almost K\"ahler structure on $M$ with generalized K\"ahler metric $G=-\mathbb{J}\mathbb{J}'$ and $B\in \Omega^2(M)$. Then the pair $(\mathbb{J},\mathbb{J}')_B:=(e^{-B}\mathbb{J}e^B,e^{-B}\mathbb{J}'e^B)$ is again a generalized almost  K\"ahler structure on $M$ and its respective generalized K\"ahler metric is given by $G_B=e^{-B}Ge^B$. We will refer to the action $e^B\cdot (\mathbb{J},\mathbb{J}')=(\mathbb{J},\mathbb{J}')_B$ as diagonal action.
\end{remark}

\begin{proposition}
A generalized K\"ahler structure on $\mathfrak u_\alpha$  is given by  a pair  $(\mathbb{J}_\alpha,\mathbb{J}'_\alpha)$ where 
$\mathbb J_\alpha$ is of complex type, and $\mathbb{J}'_\alpha$ is of noncomplex type (or vice-versa). 
\end{proposition}

\begin{proof}
Let $\alpha$ be a positive root.
\begin{enumerate}
\item[$\iota.$] If $\mathbb{J}_\alpha$ and $\mathbb{J}'_\alpha$ are both of complex type we get that $G_\alpha=\pm I_4$. Note that, by Remark \ref{SignatureKahler}, $G_\alpha$ must have signature $(2,2)$. Thus $I_4$ does not induce a generalized metric on $\mathfrak{u}_\alpha \oplus \mathfrak{u}_\alpha^*$,  similarly for $-I_4$.

\item[$\iota\iota.$] Let $\mathbb{J}_{\alpha}$ and $\mathbb{J}_{\alpha}'$ be of noncomplex type and suppose that
$$\mathcal{J}_{\alpha}={\tiny\left( 
	\begin{array}{cccc}
	a_\alpha & 0 & 0 & -x_\alpha\\ 
	0 &  a_\alpha& x_\alpha & 0\\
	0 & -y_\alpha & -a_\alpha & 0\\
	y_\alpha & 0 & 0 & -a_\alpha
	\end{array}%
	\right)} \quad \textnormal{and}\quad \mathcal{J}_{\alpha}'={\tiny\left( 
	\begin{array}{cccc}
	b_\alpha & 0 & 0 & -w_\alpha\\ 
	0 &  b_\alpha& w_\alpha & 0\\
	0 & -z_\alpha & -b_\alpha & 0\\
	z_\alpha & 0 & 0 & -b_\alpha
	\end{array}%
	\right)},$$
where $a_\alpha^2=x_\alpha y_\alpha-1$ and $b_\alpha^2=w_\alpha z_\alpha-1$. Then we have that
$\mathcal{J}_{\alpha}\mathcal{J}_{\alpha}'=\mathcal{J}_{\alpha}'\mathcal{J}_{\alpha}$ if and only if
\begin{equation} \label{comm}
\left\lbrace \begin{array}{lc} 
x_\alpha z_\alpha=w_\alpha y_\alpha\\
b_\alpha x_\alpha=a_\alpha w_\alpha.
\end{array}
\right.
\end{equation}
Since $a_\alpha^2=x_\alpha y_\alpha-1$ and $b_\alpha^2=w_\alpha z_\alpha-1$, 
we have that $x_\alpha,y_\alpha,z_\alpha,w_\alpha$ are nonzero real numbers. 
Thus, the system \eqref{comm} has possible solutions:
\begin{enumerate}
	\item[{\bf S}.] $a_\alpha=0$, $b_\alpha=0$ and $z_\alpha =\dfrac{w_\alpha y_\alpha}{x_\alpha}$ or
	\item[{\bf NC}.] $b_\alpha\neq 0$, $a_\alpha\neq 0$, $x_\alpha=\dfrac{a_\alpha w_\alpha}{b_\alpha}$ and $z_\alpha=\dfrac{b_\alpha y_\alpha}{a_\alpha}$.
\end{enumerate}
Therefore we get that $G_\alpha=-\mathcal{J}_{\alpha}\mathcal{J}_{\alpha}'=(w_\alpha y_\alpha -a_\alpha b_\alpha) I_4$, and 
this does not satisfy the requirement of signature $(2,2)$ for being a generalized metric from Remark \ref{SignatureKahler}.


\item[$\iota \iota \iota.$] Let us now consider a pair $(\mathbb{J}_{\alpha},\mathbb{J}_{\alpha}')$ where 
$\mathbb J_\alpha$  is of complex type and the $\mathbb J'_\alpha$ is of noncomplex type. 
Assume first that $\mathbb{J}_{\alpha}=\pm \mathcal{J}_0$ and $\mathbb{J}_\alpha'$ is of symplectic type ($a_\alpha=0$), that is,
$\mathbb{J}_\alpha'=\left(
\begin{array}{cc}
0 & -\omega_\alpha^{-1}\\
\omega_\alpha & 0
\end{array}%
\right)$
where $\omega_\alpha=\left( 
\begin{array}{cc}
0 & -1/x_\alpha \\
1/x_\alpha  & 0
\end{array}%
\right)=-\dfrac{k_\alpha}{x_\alpha}\cdot \omega_{b_0}|_{\mathfrak{u}_\alpha}$. It then follows that
$$G_\alpha=-(\pm \mathcal{J}_0)\mathcal{J}'_\alpha=\pm \left(
\begin{array}{cc}
0 & g_\alpha^{-1}\\
g_\alpha & 0
\end{array}%
\right)$$
where $g_\alpha=\left( 
\begin{array}{cc}
1/x_\alpha & \\
0  & 1/x_\alpha
\end{array}%
\right)=\dfrac{1}{x_\alpha}\cdot I_2$, which in fact has the form of
Example \ref{typical}.

In this case, we conclude that:
\begin{enumerate}
\item[$\iota\iota\iota_A.$] a pair $(\mathcal{J}_0,\mathcal{J}_\alpha')$, where $\mathcal{J}_\alpha'$ is of symplectic type, defines a generalized K\"ahler structure on $\mathfrak{u}_\alpha$ if and only if $x_\alpha>0$, and 
\item[$\iota\iota\iota_B.$] a pair $(-\mathcal{J}_0,\mathcal{J}_\alpha')$, where $\mathcal{J}_\alpha'$ is of symplectic type, defines a generalized K\"ahler structure on $\mathfrak{u}_\alpha$ if and only if $x_\alpha<0$.
\end{enumerate}
Consider now the more general situation when  $(\pm \mathcal J_0,\mathbb{J}_{\alpha}')$ where 
$\mathbb{J}_{\alpha}'$ is of noncomplex type. Using Remark \ref{diagonalBaction},
we suppose that $\mathbb{J}_{\alpha}'={\tiny\left( 
	\begin{array}{cccc}
	a_\alpha & 0 & 0 & -x_\alpha\\ 
	0 &  a_\alpha& x_\alpha & 0\\
	0 & -y_\alpha & -a_\alpha & 0\\
	y_\alpha & 0 & 0 & -a_\alpha
	\end{array}%
	\right)}$ with $a_\alpha^2=x_\alpha y_\alpha -1$ and let $\mathbb{J}_{\omega_\alpha}'$ be its associated generalized complex structure of symplectic type; see Notation \ref{SymplecticAsso}. By Proposition \ref{BFieldSymplectic} we know that $\mathbb{J}_{\alpha}'=e^{-B_\alpha}\mathbb{J}_{\omega_\alpha}'e^{B_\alpha}$ where $B_\alpha=-\dfrac{a_\alpha}{x_\alpha}S_\alpha^\ast\wedge A_\alpha^\ast$. As the action by $B$-transformations fixes the generalized complex structures of complex type, we get that
\begin{equation}\label{diagonalBaction0}
e^{B_\alpha}\cdot (\pm \mathcal{J}_0,\mathbb{J}_{\omega_\alpha}')=(\pm \mathcal{J}_0,\mathbb{J}_{\alpha}').
\end{equation}
Moreover,
$$G_\alpha=-(\pm \mathcal{J}_0)\mathcal{J}_\alpha=\pm\left( 
\begin{array}{cccc}
0 & a_\alpha & x_\alpha & 0\\ 
-a_\alpha & 0& 0 & x_\alpha\\
y_\alpha & 0 & 0 & -a_\alpha\\
0 & y_\alpha & a_\alpha & 0
\end{array}%
\right),$$
with $a^2=x_\alpha y_\alpha-1$. Thus $G_\alpha$ is a generalized K\"ahler metric, according to Remark \ref{diagonalBaction}.

\end{enumerate}
\end{proof}

As a consequence, it then follows that
the general case behaves similarly to  Example \ref{typical}.

\begin{notation}
Let $\mathbb{J}_{\alpha}$ be of noncomplex type with $\mathcal{J}_{\alpha}={\tiny\left( 
	\begin{array}{cccc}
	a_\alpha & 0 & 0 & -x_\alpha\\ 
	0 &  a_\alpha& x_\alpha & 0\\
	0 & -y_\alpha & -a_\alpha & 0\\
	y_\alpha & 0 & 0 & -a_\alpha
	\end{array}%
	\right)}$ where $a_\alpha^2=x_\alpha y_\alpha -1$. If $x_\alpha>0$, we denote $\mathbb{J}_{\alpha}$ by $\mathcal{J}_{\alpha}^+$. Otherwise, we denote it by $\mathcal{J}_{\alpha}^-$.
\end{notation}
Summing up,
\begin{corollary}\label{Kahlerroots}
Let $(\mathbb{J},\mathbb{J}')$ be an invariant generalized almost K\"ahler structure on $\mathbb{F}$. Then, for all $\alpha\in \Pi^+$, 
the pair  $(\mathbb{J}_\alpha,\mathbb{J}_\alpha')$ takes one of the following values
\begin{center}
	\begin{tabular}{c|c}
		$\mathbb{J}_\alpha$ & $\mathbb{J}'_\alpha$\\
		\hline
		$\mathcal{J}_0$ & $(\mathcal{J}'_\alpha)^+$ \\
		$\mathcal{J}_\alpha^+$ & $\mathcal{J}_0$ \\
		$-\mathcal{J}_0$ & $(\mathcal{J}_\alpha')^- $  \\
		$\mathcal{J}_\alpha^-$ & $-\mathcal{J}_0$.
	\end{tabular}
\end{center}
\end{corollary}
Motivated by Remark \ref{diagonalBaction} we define:
\begin{definition}\label{BmoduliK1}
The moduli space of generalized almost K\"ahler structures on $M$ is defined as the quotient space $\mathfrak{K}_a:=\dfrac{\mathcal{K}_a}{\mathcal{B}}$ determined by the diagonal action of $\mathcal{B}$ on $\mathcal{K}_a$. Analogously, we define the moduli space of generalized  K\"ahler structures on $M$ as the quotient space $\mathfrak{K}=\dfrac{\mathcal{K}}{\mathcal{B}_c}$ determined by the diagonal action of $\mathcal{B}_c$ on $\mathcal{K}$.
\end{definition}
Thus, by Proposition \ref{Kahlerroots} and arguments used to deduce \eqref{diagonalBaction0} we get:
\begin{corollary}
Let $\alpha$ be a positive root. The set  $\displaystyle \mathfrak{K}_\alpha(\mathbb{F})=\frac{\mathcal K_\alpha}{B}$ of equivalence classes of invariant generalized almost K\"ahler structures on $\mathfrak{u}_\alpha$ modulo 
the diagonal action of $B$-transformations is given by:
$$\mathfrak{K}_\alpha(\mathbb{F})=(\lbrace \mathbf{0}\rbrace \times \mathbb{R}^+)\cup (\mathbb{R}^+\times \lbrace \mathbf{0}\rbrace)\cup (\lbrace -\mathbf{0}\rbrace \times \mathbb{R}^-)\cup (\mathbb{R}^-\times \lbrace -\mathbf{0}\rbrace).$$
\end{corollary}

The natural topology on $\mathfrak{K}_\alpha(\mathbb{F})$ as subspace of $\mathbb R^2$  coincides with the  topology induced as subspace   from $\mathfrak{M}_\alpha(\mathbb{F})\times \mathfrak{M}_\alpha(\mathbb{F})$. In further  generality, if $\mathbb{R}^\dag:=(\lbrace \mathbf{0}\rbrace \times \mathbb{R}^+)\cup (\mathbb{R}^+\times \lbrace \mathbf{0}\rbrace)\cup (\lbrace -\mathbf{0}\rbrace \times \mathbb{R}^-)\cup (\mathbb{R}^-\times \lbrace -\mathbf{0}\rbrace)$ (the union of the four rays) we have that
\begin{theorem}\label{BmoduliK2}
	Suppose that $\vert \Pi^+ \vert=d$. Then
	$$\mathfrak K_a (\mathbb F)= \prod_{\alpha\in\Pi^+} \mathfrak{K}_\alpha(\mathbb{F})=\mathbb{R}^\dag_{\alpha_1}\times \cdots \times\mathbb{R}^\dag_{\alpha_d}=(\mathbb{R}^\dag)^d.$$
\end{theorem}

The topology on $\mathfrak K_a (\mathbb F)$ inherited from $\mathbb R^n$ coincides  with the   topology induced from the product 
 $\mathfrak{M}_a(\mathbb{F})\times \mathfrak{M}_a(\mathbb{F})$.

\subsection{Integrability}
By Definition \ref{ke}, when we require that both $\mathbb{J}$ and $\mathbb{J}'$ be integrable we get that the pair $(\mathbb{J},\mathbb{J}')$ determines an invariant generalized K\"ahler structure on $\mathbb{F}$. Thus, to determine which of those invariant generalized almost K\"ahler structures given in Proposition \ref{Kahlerroots} are in fact invariant generalized K\"ahler structures we need to use Table \ref{integrab}.

\begin{remark}\label{triplekahler}
Let $(\mathbb{J},\mathbb{J}')$ be an invariant generalized Kähler structure on $\mathbb{F}$. Since $\mathbb{J}$ is an invariant generalized complex structure on $\mathbb{F}$, then for each triple of roots $(\alpha,\beta,\alpha+\beta)$, we have that the only possibilities for $(\mathbb{J}_\alpha, \mathbb{J}_\beta, \mathbb{J}_{\alpha+\beta})$ are those described in Table \ref{integrab}. Since for each positive root $\alpha$, the pair $(\mathbb{J}_\alpha,\mathbb{J}_\alpha ')$ must be formed by one structure of complex type and the other one of noncomplex type, we observe that:
\begin{itemize}
 \item the possible values for  $\pm(\mathbb{J}_\alpha, \mathbb{J}_\beta,\mathbb{J}_{\alpha+\beta})$ 
 are:
 
 either  $ (\mathcal{J}_\alpha^{+},\mathcal{J}_0,\mathcal{J}_0)$, or $(\mathcal{J}_0,\mathcal{J}_\beta^{+},\mathcal{J}_0)$, or else  $ (\mathcal{J}_0,-\mathcal{J}_0,\mathcal{J}_{\alpha+\beta}^{+})$ 

\item with corresponding values for 
$\pm (\mathbb{J}' _\alpha, \mathbb{J}' _\beta,\mathbb{J}' _{\alpha+\beta})$, respectively:

either  $ (\mathcal{J}_0,(\mathcal{J}_\beta')^{+},(\mathcal{J}_{\alpha+\beta}')^{+})$, or  $ ((\mathcal{J}_\alpha')^{+} ,\mathcal{J}_0,(\mathcal{J}_{\alpha+\beta}')^{+})$, or else  $((\mathcal{J}_\alpha')^{+},(\mathcal{J}_\beta')^{-},\mathcal{J}_0)$.
\end{itemize}
 However, note that for these cases $\mathbb{J}'$ is not integrable. Therefore, in the other cases of Table \ref{integrab}, if we have that $(\mathbb{J}_\alpha, \mathbb{J}_\beta,\mathbb{J}_{\alpha+\beta})$ is purely of complex type (resp. of noncomplex type), then $(\mathbb{J}' _\alpha, \mathbb{J}' _\beta,\mathbb{J}' _{\alpha+\beta})$ must be purely of noncomplex type (resp. of complex type).
\end{remark}

\begin{lemma}\label{simplerootK}
Let $(\mathbb{J},\mathbb{J} ')$ be an invariant generalized Kähler structure on $\mathbb{F}$ and $\Sigma$ be a simple root system for $\mathfrak{g}$. If $\mathbb{J}_{\alpha_{i}}$ is of noncomplex type for some simple root $\alpha_i\in \Sigma$, then $\mathbb{J}_{\alpha_j}$ is of noncomplex type for all simple roots $\alpha_j\in \Sigma$. 
\end{lemma}
\begin{proof}
We separate the proof in 2 cases according to types of the  semisimple Lie algebra. Let $\mathfrak{g}$ be a semisimple Lie algebra of type $A$, $B$ or $C$. If $\Sigma = \{ \alpha_1, \alpha_2, \cdots, \alpha_l\}$ is a simple root system, then we get that the roots of height 2 are given by $\alpha_1+\alpha_2$, $\alpha_2+\alpha_3$, $\cdots$, $\alpha_{l-1}+\alpha_l$. Suppose that $\mathbb{J}_{\alpha_i}$ is of noncomplex type for some $\alpha_i\in \Sigma$. By Remark \ref{triplekahler}, given the triple of positive roots $(\alpha_i,\alpha_{i+1},\alpha_i+\alpha_{i+1})$, we have that $(\mathbb{J}_{\alpha_i}, \mathbb{J}_{\alpha_{i+1}},\mathbb{J}_{\alpha_i + \alpha_{i+1}})$ is purely of noncomplex type. Again, since $\mathbb{J}_{\alpha_{i+1}}$ is of noncomplex type as well and we have the triple of positive roots $(\alpha_{i+1},\alpha_{i+2},\alpha_{i+1}+\alpha_{i+2})$, by Remark \ref{triplekahler}, we get that $(\mathbb{J}_{\alpha_{i+1}}, \mathbb{J}_{\alpha_{i+2}},\mathbb{J}_{\alpha_{i+1} + \alpha_{i+2}})$ is purely of noncomplex type. The same reasoning can be used with the triple of positive roots $(\alpha_{i-1},\alpha_i,\alpha_{i-1}+\alpha_i)$. Therefore, an iterative process allows us to show that for all $\alpha_j \in \Sigma$ we obtain that $\mathbb{J}_{\alpha_j}$ is of noncomplex type. The same arguments can be used for the Lie algebras of type $F_4$ and $G_2$.

Let us now assume that $\mathfrak{g}$ is a semisimple Lie algebra of type $D$. If $\Sigma = \{ \alpha_1, \alpha_2, \cdots, \alpha_l\}$ is a simple root system,  in case of type $D$ we have that the roots of height 2 are $\alpha_1+\alpha_2$, $\alpha_2+\alpha_3$, $\cdots$, $\alpha_{l-2}+\alpha_{l-1}$, $\alpha_{l-2} +\alpha_{l}$. Suppose that $\mathbb{J}_{\alpha_i}$ is of noncomplex type for some $\alpha_{i}\in \Sigma$. Note that we just need to consider separately  if $\alpha_i$ is either one of the roots $\alpha_{l-1}$ or $\alpha_{l}$, because the other cases are analogous to the previous one. Suppose that $\mathbb{J}_{\alpha_i}=\mathbb{J}_{\alpha_{l-1}}$. Thus, given the triple of positive roots $(\alpha_{l-2},\alpha_{l-1},\alpha_{l-2}+\alpha_{l-1})$, as consequence of Remark \ref{triplekahler} we have that $(\mathbb{J}_{\alpha_{l-2}},\mathbb{J}_{\alpha_{l-1}},\mathbb{J}_{\alpha_{l-2}+\alpha_{l-1}})$ must be purely of noncomplex type. By the same reason, if we take the triple of positive roots $(\alpha_{l-2},\alpha_{l},\alpha_{l-2}+\alpha_{l})$ then $(\mathbb{J}_{\alpha_{l-2}},\mathbb{J}_{\alpha_{l}},\mathbb{J}_{\alpha_{l-2}+\alpha_{l}})$ must be purely of noncomplex type. Therefore, as $\mathbb{J}_{\alpha_{l-2}}$ is of noncomplex type, using the same arguments as we did in the case $A$, we get that $\mathbb{J}_{\alpha_{j}}$ for all $j=1,\cdots, l-3$ are of noncomplex type. For the case where $\mathbb{J}_{\alpha_l}$ is of noncomplex type the arguments are analogous. The cases of Lie algebras of type $E_6$, $E_7$ and $E_8$ can be reasoned as we did with the case $D$.
\end{proof}

\begin{lemma}\label{simplerootK2}
Let $(\mathbb{J},\mathbb{J} ')$ be an invariant generalized Kähler structure on $\mathbb{F}$ and $\Sigma$ be a simple root system for $\mathfrak{g}$. If $\mathbb{J}_{\alpha_i}$ is of noncomplex type for some simple root $\alpha_i\in \Sigma$, then $\mathbb{J}_{\alpha}$ is of noncomplex type for every positive root $\alpha\in \Pi^+$. 
\end{lemma}
\begin{proof}
As consequence of Lemma \ref{simplerootK}, we have that if  $\mathbb{J}_{\alpha_i}$ is of noncomplex type for some simple root $\alpha_i\in \Sigma$, then $\mathbb{J}_{\alpha_j}$ is of noncomplex type for all simple root $\alpha_j\in \Sigma$. If $\alpha$ is a positive root of height $n$, it can be written as $\alpha = \beta +\gamma$, where $\beta$ is a root of height $1$ (that is, a simple root) and $\gamma$ is a root of height $n-1$. Thus, we get a triple of positive roots $(\beta, \gamma,\beta+\gamma=\alpha)$.  Since $\mathbb{J}_\beta$ is of noncomplex type, by Remark \ref{triplekahler}, we have that $(\mathbb{J}_\beta, \mathbb{J}_\gamma, \mathbb{J}_\alpha)$ must be purely of noncomplex type. Therefore, $\mathbb{J}_\alpha$ is of noncomplex type for all $\alpha\in \Pi^+$, as desired.
\end{proof}

\begin{theorem}\label{KahlerClassification}
Let $(\mathbb{J},\mathbb{J} ')$ be an invariant generalized Kähler structure on $\mathbb{F}$. If $\mathbb{J}_{\alpha}$ is of noncomplex type for some positive root $\alpha$, then $\mathbb{J}_{\beta}$ is of noncomplex type for every positive root $\beta$. In particular, every invariant generalized Kähler structure on $\mathbb{F}$  consists  of a complex structure and a $B$-transformation of a symplectic structure.
\end{theorem}
\begin{proof}
Let $\alpha$ be a positive root such that $\mathbb{J}_\alpha$ is of noncomplex type. If $\alpha$ is a simple root, then the result immediately follows from Lemma \ref{simplerootK2}. Suppose now that $\alpha$ is a positive root with height $n \geq 2$. Thus, we have that $\alpha$ can be written as $\alpha = \beta + \gamma$, where $\beta$ is a root of height $1$ (that is, a simple root) and $\gamma$ is a root of height $n-1$. Therefore, we get a triple of positive roots $(\beta, \gamma,\beta+\gamma=\alpha)$. Since $\mathbb{J}_\alpha$ is of noncomplex type, by Remark \ref{triplekahler} we deduce that $(\mathbb{J}_\beta, \mathbb{J}_\gamma, \mathbb{J}_\alpha)$ must be purely of noncomplex type. Hence, $\mathbb{J}_\beta$ is of noncomplex type and given that $\beta$ is a simple root the result follows from Lemma \ref{simplerootK2}.

In particular, using this fact and Theorem \ref{GlobalBField} it is simple to conclude that every invariant generalized Kähler structure on $\mathbb{F}$ is always formed by a structure of complex type and a $B$-transformation of a structure of symplectic type.
\end{proof}

Following the discussion in Section \ref{S:4}, we can define a natural action of the Weyl group $\mathcal{W}$ on the set of all invariant generalized almost K\"ahler structures on $\mathbb{F}$ as follows. For each $w\in \mathcal{W}$ and $(\mathbb{J},\mathbb{J}')\in\mathcal{K}_a$ we define
$$w\cdot (\mathbb{J},\mathbb{J}')=(w\cdot \mathbb{J},w\cdot \mathbb{J}').$$
It is simple to check that $w\cdot (\mathbb{J},\mathbb{J}')$ is  well defined, that is, it defines another invariant generalized K\"ahler structure on $\mathbb{F}$ and by Proposition \ref{WeylIntegrability}:
\begin{proposition}\label{WeylKahler}
	Let $\Sigma$ be a simple root system for $\mathfrak{g}$ with $\Pi^+$ the corresponding set of positive roots. If $(\mathbb{J},\mathbb{J}')$ is an invariant generalized K\"ahler structure on $\mathbb{F}$, then $w\cdot (\mathbb{J},\mathbb{J}')$ defines another invariant generalized K\"ahler structure on $\mathbb{F}$.
\end{proposition}
Let $\alpha$ be a positive root. Recall that if $\mathbb{J}_\alpha=\pm\mathcal{J}_0$ is of complex type then $\mathbb{J}_{-\alpha}=\mp\mathcal{J}_0$ and if $\mathbb{J}_\alpha=\mathcal{J}_\alpha$ is of noncomplex type with
$\mathcal{J}_\alpha={\tiny \left( 
	\begin{array}{cccc}
	a_\alpha & 0 & 0 & -x_\alpha\\ 
	0 &  a_\alpha& x_\alpha & 0\\
	0 & -y_\alpha & -a_\alpha & 0\\
	y_\alpha & 0 & 0 & -a_\alpha
	\end{array}%
	\right)}$ then $\mathbb{J}_{-\alpha}$ is also of noncomplex type where $a_{-\alpha}=a_\alpha$, $x_{-\alpha}=-x_\alpha$, and $y_{-\alpha}=-y_\alpha$ (see \cite{VS}). Thus, up to action by the Weyl group $\mathcal{W}$, an invariant generalized almost complex structure $\mathbb{J}$ on $\mathbb{F}$ satisfies:
\begin{enumerate}
	\item[$\iota.$] if $\mathbb{J}_\alpha$ is of noncomplex type then $x_\alpha>0$, or
	\item[$\iota\iota.$] if $\mathbb{J}_\alpha$ is of complex type then $\mathbb{J}_\alpha=\mathcal{J}_0$.
\end{enumerate}
Therefore, up to diagonal action by the Weyl group, we have that 
$$\mathfrak{K}_\alpha(\mathbb{F})/\mathcal{W}=(\lbrace \mathbf{0}\rbrace \times \mathbb{R}^+)\cup (\mathbb{R}^+\times \lbrace \mathbf{0}\rbrace).$$
So,
\begin{proposition}
	Suppose that $\vert \Pi^+ \vert=d$. Then
	$$\mathfrak K_a (\mathbb F)/\mathcal{W}=((\lbrace \mathbf{0}\rbrace \times \mathbb{R}^+)\cup (\mathbb{R}^+\times \lbrace \mathbf{0}\rbrace))_{\alpha_1}\times \cdots \times ((\lbrace \mathbf{0}\rbrace \times \mathbb{R}^+)\cup (\mathbb{R}^+\times \lbrace \mathbf{0}\rbrace))_{\alpha_d}.$$
\end{proposition}

\subsection{Moduli space of generalized metrics}
We now discuss invariant generalized metrics on $\mathbb{F}$. Recall that if $(\mathbb{J},\mathbb{J}')$ is a generalized K\"ahler structure on a manifold $M$ with generalized metric $G$ and $B\in \Omega^2(M)$ then $e^B\cdot (\mathbb{J},\mathbb{J}')$ is another generalized K\"ahler structure if and only if $B$ is closed. Moreover, the generalized metric associated to $e^B\cdot (\mathbb{J},\mathbb{J}')$ is given by $G_B=e^{-B}G e^B$. If we denote the set of all invariant generalized metric on $\mathbb{F}$ by $\mathcal{G}$, we get that the expression $e^B\cdot G:=e^{-B}G e^B$ allows us to obtain a well-defined action of $\mathcal{B}$ on $\mathcal{G}$. 
\begin{definition}
The {\it moduli space of generalized metrics} $\mathfrak{G}$ on $M$ is defined as the quotient space $\mathfrak{G}:=\mathcal{G}/\mathcal{B}$ determined by the action of $\mathcal{B}$ on $\mathcal{G}$.
\end{definition}
Let $\alpha$ be a positive root and $(\mathbb{J},\mathbb{J}')$ be an invariant generalized  K\"ahler structure on $\mathbb{F}$. Thus, $(\mathbb{J}_\alpha,\mathbb{J}'_\alpha)$ is formed by one structure of complex type and the other one of noncomplex type. Without loss of generality, let us assume that $\mathbb{J}_{\alpha}=\pm \mathcal{J}_0$ and $\mathbb{J}_{\alpha}'={\tiny\left( 
	\begin{array}{cccc}
	a_\alpha & 0 & 0 & -x_\alpha\\ 
	0 &  a_\alpha& x_\alpha & 0\\
	0 & -y_\alpha & -a_\alpha & 0\\
	y_\alpha & 0 & 0 & -a_\alpha
	\end{array}%
	\right)}$ with $a_\alpha^2=x_\alpha y_\alpha -1$ and let $\mathbb{J}_{\omega_\alpha}'$ be its associated generalized complex structure of symplectic type; see Notation \ref{SymplecticAsso}. Then, for $B_\alpha=-\dfrac{a_\alpha}{x_\alpha}S_\alpha^\ast\wedge A_\alpha^\ast$, we have that
$$G_\alpha=-(\pm \mathcal{J}_0)\mathcal{J}_\alpha=\pm\left( 
\begin{array}{cccc}
0 & a_\alpha & x_\alpha & 0\\ 
-a_\alpha & 0& 0 & x_\alpha\\
y_\alpha & 0 & 0 & -a_\alpha\\
0 & y_\alpha & a_\alpha & 0
\end{array}%
\right)=e^{-B_\alpha} \left(\pm\left(
\begin{array}{cc}
0 & g_\alpha^{-1}\\
g_\alpha & 0
\end{array}%
\right)\right)e^{B_\alpha} $$
where $g_\alpha=\left( 
\begin{array}{cc}
1/x_\alpha & \\
0  & 1/x_\alpha
\end{array}%
\right)=\dfrac{1}{x_\alpha}\cdot I_2$ with the conditions that $x_\alpha>0$ if $\mathbb{J}_{\alpha}=\mathcal{J}_0$ and $x_\alpha<0$ otherwise. Therefore, the set  $\displaystyle \mathfrak{G}_\alpha(\mathbb{F})=\frac{\mathcal G_\alpha}{B}$ of equivalence classes of invariant generalized metrics on $\mathfrak{u}_\alpha$ modulo 
the action of $B$-transformations is given by
$$\mathfrak{G}_\alpha(\mathbb{F})=\mathbb{R}^\ast.$$
More importantly, since every invariant generalized metric can be written as $G=\displaystyle \bigoplus_{\alpha\in \Pi^+}G_\alpha$, we  get

\begin{proposition}\label{GereralizedMetrics}
Suppose that $\vert \Pi^+ \vert=d$. Then
$$\mathfrak G (\mathbb F)= \prod_{\alpha\in\Pi^+} \mathfrak{G}_\alpha(\mathbb{F})=\mathbb{R}^\ast_{\alpha_1}\times \cdots \times\mathbb{R}^\ast_{\alpha_d}=(\mathbb{R}^\ast)^d.$$
Moreover, up to action by the Weyl group $\mathcal{W}$, we have that $\mathfrak G (\mathbb F)/\mathcal{W}=(\mathbb{R}^+)^d$.
\end{proposition}

\section{Examples: the cases $\mathfrak{sl}(2,\mathbb{C})$ and $\mathfrak{sl}(3,\mathbb{C})$}\label{S:7}
In this section we are interested in providing two elementary examples using the root space decomposition of 
$\mathfrak{sl}(n,\mathbb{C})$; see Example \ref{CaseSln}.
\begin{example}
	In $\mathfrak{sl}(2,\mathbb{C})$ we know that $\Pi=\lbrace \alpha,-\alpha \rbrace$. For this case we have that $\Sigma=\Pi^+=\lbrace \alpha\rbrace$ and $\mathbb{F}=\mathbb{P}^1$. All invariant generalized almost complex structures on $\mathbb{P}^1$ are determined using  either the usual complex structure or the KKS symplectic form $\omega$ plus a $B$-transformation. Moreover, all these structures are integrable (see \cite{VS}).
	\begin{enumerate}
		\item[$\iota.$] If $\Theta=\emptyset$ we have that $\mathbb{J}=\pm \mathcal{J}_0$.
		
		\item[$\iota\iota.$] If $\Theta=\lbrace \alpha \rbrace$ we have that $$\mathbb{J}={\tiny\left( 
			\begin{array}{cccc}
			a_\alpha & 0 & 0 & -x_\alpha\\ 
			0 &  a_\alpha& x_\alpha & 0\\
			0 & -y_\alpha & -a_\alpha & 0\\
			y_\alpha & 0 & 0 & -a_\alpha
			\end{array}%
			\right)},$$
		with $x_\alpha,y_\alpha,a_\alpha\in\mathbb{R}$ such that $a^2_\alpha=x_\alpha y_\alpha-1$. The invariant pure spinor line is given by $\varphi=e^{B_\alpha+i\omega_\alpha}$ with $\omega_\alpha=-\dfrac{k_\alpha}{x_\alpha}\omega_{b_0}$ and $B_\alpha=-\dfrac{a_\alpha}{x_\alpha}S_\alpha^\ast\wedge A_\alpha^\ast$.
	\end{enumerate}
	The module space $\mathfrak{M}(\mathbb{P}^1)=\mathbb{R}^\ast\cup \pm \mathbf{0}$. It is simple to check that up to action by the Weyl group there are two different classes of invariant generalized complex structures on $\mathbb{P}^1$, namely, $\mathcal{J}_0$ and $\mathcal{J}_\alpha$ with $x_\alpha>0$. Thus, $\mathfrak{M}(\mathbb{P}^1)/\mathcal{W}=(0,\infty)\cup \mathbf{0}$ and its compactification can be identified with $S^1$. Moreover, $\mathfrak{K}(\mathbb{P}^1)=\mathbb{R}^\dag$ and $\mathfrak{G}(\mathbb{P}^1)=\mathbb{R}^\ast$.
\end{example}

\begin{example}
Consider $\mathfrak{sl}(3,\mathbb{C})$. In this case we have that $\Pi=\lbrace \alpha,\beta ,\alpha+\beta,-\alpha, -\beta,-\alpha-\beta\rbrace$ and $\mathbb{F}$ is given by $\mathbb{F}(1,2)\subset \mathbb{P}^2\times \mathbb{P}^2$ where
$$\mathbb{F}(1,2)=\lbrace ([x_1:x_2:x_3],[y_1:y_2:y_3])\in \mathbb{P}^2\times \mathbb{P}^2:\ x_1y_1+x_2y_2+x_3y_3=0 \rbrace.$$
There are $27$ possible configurations of invariant generalized almost complex structures on $\mathbb{F}(1,2)$ where only 13 of these, which are described by Table \ref{integrab}, are integrable. Up to action by the Weyl group we have just the following cases 
\begin{center}
	\begin{tabular}{c|c|c}
		$\mathbb{J}_\alpha$ & $\mathbb{J}_\beta$ & $\mathbb{J}_{\alpha + \beta}$ \\
		\hline
		$\mathcal{J}_0$ & $\mathcal{J}_0$ & $\mathcal{J}_0$\\
		$ \mathcal{J}_\alpha$ & $\mathcal{J}_0 $ & $ \mathcal{J}_0$ \\
		$\mathcal{J}_\alpha$ & $\mathcal{J}_\beta$ & $\mathcal{J}_{\alpha+\beta} $
	\end{tabular}
\end{center}
where if $\mathbb{J}_\alpha$ is of noncomplex type then $x_\alpha>0$; see Remark \ref{x>0}. The moduli space $$\mathfrak{M}_a(\mathbb{F}(1,2))=(\mathbb{R}^\ast\cup \pm \mathbf{0})_\alpha \times( \mathbb{R}^\ast\cup \pm \mathbf{0})_\beta \times( \mathbb{R}^\ast\cup \pm \mathbf{0})_{\alpha+\beta},$$ 
and the invariant pure spinor are determined by Proposition \ref{Spinor}. Also, the moduli space of generalized almost Kähler structures and metrics are $$\mathfrak{K}_a(\mathbb{F}(1,2)) = \mathbb{R}^\dag _{\alpha}\times \mathbb{R}^\dag _{\beta} \times \mathbb{R}^\dag _{\alpha+\beta}\qquad\textnormal{and}\qquad \mathfrak{G}(\mathbb{F}(1,2))=(\mathbb{R}^\ast)^3,$$
respectively.

\end{example}




\section*{Acknowledgements}
We started this work during a meeting of the Working Group on Deformation Theory 
at the Abdus Salam International Center for Theoretical Physics (ICTP), Italy, in August 2019. 
We are grateful to both the Mathematics and the High Energy Sections for making our meeting possible. 
The authors also benefitted from meetings of  Network NT8 for Geometry and Physics, 
from the Office of External Activities, ICTP. We are grateful to Gil Cavalcanti for  valuable suggestions 
which helped us make corrections to the previous version of our work. 

E.G. was partially supported by a Simons Associateship ICTP.  

F.V.  thanks the support  Vicerrector\'ia de Investigaci\'on, Universidad Cat\'olica del Norte; Chile. 

C.V. was supported by FAPESP grant 2016/07029-2 and CAPES Brazil - Finance Code 001.


\begin{thebibliography}{WWW}
\bibitem[AD]{AD}  Alekseevsky, D.; David, L., {\it Invariant generalized complex structures on Lie groups}, Proc. Lond. Math. Soc. (3),  {\bf 105} (2012) n. 4, 703--729. 

\bibitem[BLPZ]{BLPZ}  Bredthauer, A.; Lindstr\"om, U.; Persson, J.; Zabzine, M., {\it Generalized K\"ahler geometry from supersymmetric sigma models}, Lett. Math. Phys.,  {\bf 77} (2006) n. 3, 291--308. 
\bibitem[Ca1]{Ca}  Cavalcanti, G., {\it The decomposition of forms and cohomology of generalized complex manifolds}, J. Geom. Phys.,  {\bf 57} (2007) n. 1, 121--132.
\bibitem[Ca2]{Ca2}  Cavalcanti, G., {\it New aspects of the {$dd^c$}-lemma}, D.Phil. Thesis, Oxford University, 2004. 
\bibitem[CG1]{CG}  Cavalcanti, G.; Gualtieri, M., {\it Generalized complex structures on nilmanifolds}, J. of Symplectic Geo.,  {\bf 2} (2004) n. 3, 393--410.
\bibitem[CG2]{CG2}  Cavalcanti, G.; Gualtieri, M., {\it Generalized complex geometry and {$T$}-duality}, A celebration of the mathematical legacy of Raoul Bott, CRM Proc. Lecture Notes, Amer. Math. Soc., Providence, RI, {\bf 50} (2010), 341--365. 
\bibitem[C]{C}   Chevalley, C.,  {\it The Algebraic Theory of Spinors and Clifford Algebras, Collected Works}, Springer -- Verlang, New York, {\bf 2} (1997).
\bibitem[Gr]{Gr}   Gra\~na, M., {\it Flux compactifications in string theory: a comprehensive review}, Phys. Rep.,  {\bf 423} (2006) n. 3, 91--158. 
\bibitem[G1]{G1}   Gualtieri, M., {\it Generalized complex geometry}, Ann. of Math. (2),  {\bf 174} (2011) n. 1, 75--123. 
\bibitem[G2]{G2}   Gualtieri, M., {\it Generalized K\"ahler geometry}, Comm. Math. Phys.,  {\bf 331} (2014) n. 1, 297--331.

\bibitem[G3]{G3} Gualtieri, M. {\it Generalized complex geometry}, D.Phil. Thesis, Oxford University, 2003.



\bibitem[H]{H}   Hitchin, N., {\it Generalized Calabi-Yau manifolds}, Q. J. Math.,  {\bf 54} (2003) n. 3, 281--308.   
\bibitem[SN]{SN}   San Martin, L. A. B., Negreiros, C. J. C.; {\it Invariant almost Hermitian structures on flag manifolds}, Adv. Math.,  {\bf 178} (2003) n. 2, 277--310.
\bibitem[VS]{VS}   Varea, C. A. B. ; San Martin, L. A. B.,  {\it Invariant Generalized Complex Structures on Flag Manifolds}, J. Geom. Phys., {\bf 150} (2020), 103610. 
\end{thebibliography}
\end{document}